\documentclass[]{elsarticle}

\usepackage{amsfonts}
\usepackage{amsthm}
\usepackage{amsmath}
\usepackage{amssymb}
\usepackage{mathtools} 
\usepackage{xcolor}
\usepackage{soul} 
\usepackage{mathrsfs}

\usepackage[top=1in, bottom=1in, left=1in, right=1in]{geometry}

\newtheorem{thm}{Theorem}[section]
\newtheorem{lem}{Lemma}[section]

\theoremstyle{definition}
\newtheorem{defn}{Definition}[section]

\theoremstyle{remark}
\newtheorem{remark}{Remark}[section]

\definecolor{cucol}{rgb}{0,0,0.8}
\definecolor{afcol}{rgb}{1,0,0}

\numberwithin{equation}{section}
\begin{document}

\begin{frontmatter}
	
\title{Controllability of fractional stochastic delay dynamical systems}
	
\date{}
	
 \author[]{Arzu Ahmadova}
 \ead{arzu.ahmadova@emu.edu.tr}
 \author[]{Ismail T. Huseynov}
 \ead{ismail.huseynov@emu.edu.tr}
 \author[]{Nazim I. Mahmudov\corref{cor1}}
 \ead{nazim.mahmudov@emu.edu.tr}
 \cortext[cor1]{Corresponding author}
 
	
	\address{ Department of Mathematics, Faculty of Arts and Sciences, Eastern Mediterranean University, Gazimagusa, TRNC, via Mersin 10, Turkey}

	
	
	
	\begin{abstract}
		\noindent In this paper, we consider Caputo type fractional stochastic time-delay system with permutable matrices. We derive stochastic analogue of variation of constants formula via a newly defined delayed Mittag-Leffler type matrix function. Thus, we investigate new results on existence and uniqueness of mild solutions with the help of weighted maximum norm to fractional stochastic time-delay differential equations whose coefficients satisfy standard Lipschitz conditions. The main points in the proof are to apply Ito's isometry and martingale representation theorem, and to show the notion of a coincidence between the integral equation and the mild solution. Finally, we study complete controllability results for linear and nonlinear fractional stochastic delay dynamical systems with Wiener noise. 
	\end{abstract}
		
	
\end{frontmatter}

\section{Introduction}

Over the years, many results have been investigated on the theory and applications of stochastic differential equations \cite{oksendal, ito}, \cite{prato}-\cite{liu}. In particular, \textbf{Fractional stochastic differential equations} are a generalization of differential equations by the use of fractional and stochastic calculus. Recently, fractional stochastic differential equations are intensively applied to model mathematical problems in control theory, finance, dynamics of complex systems in engineering, infectious diseases, and other areas \cite{mendes, gangaram}. Most of the results on fractional stochastic dynamical systems are limited to prove existence and uniqueness of mild solutions using fixed point theorem \cite{anh,mahmudov,arzu}. Simultaneously, fractional differential equations has become famous in the last three decades due to its capability to model mathematical tools efficiently \cite{kilbas, diethelm, samko}. Thus, fractional-order models lead to investigation of more accurate solutions in comparison with integer-order ones. It turns out that the fractional derivatives provide the heritable properties of different physical processes more precisely. 

\textbf{Fractional delay differential equations} are differential equations covering fractional derivatives and time-delays. Delay differential equations with fractional-order have achieved a great deal of attention due to their applications in science, engineering and physics using proper numerical methods and graphical tools. In recent decades, the theory of fractional delay differential equations is also well-established by means of analytical methods. Firstly,under the assumptions that $A$ and $B$ are permutable matrices, Khusainov et al. \cite{shuklin} provided a analytical representation of a solution to a linear homogeneous matrix differential equations with a constant delay in terms of infinite series. Notice that fractional analogue of the same problem was considered by Li et al. \cite{li-wang} in particular case of $ A=\Theta$. In another paper, Li et al. \cite{wang-li} introduced a concept of delayed Mittag-Leffler type matrix function via a two-parameter Mittag-Leffler function and presented finite-time stability results to nonlinear fractional delay differential equations in the same special case. Mahmudov \cite{mahmudov2} proposed a newly defined explicit formula to linear homogeneous and nonhomogeneous fractional time-delay systems via two-parameter Mittag-Leffler perturbation in the general case (i.e., $A$ and $B$ are arbitrary constant matrices). Huseynov et al. \cite{ismail} provided a new representation of a solution through a delayed analogue of three-parameter Mittag-Leffler functions under the assumptions in which $A$ and $B$ are permutable matrices. Deriving exact solution representation of fractional delay dynamical system is a way of applying Laplace transform method and variation of constants formula to study stability, controllability, reachability and stabilizability due to their applications in control theory, chaos and bioengineering. Numerical methods for fractional time-delay systems are more intensively studied than analytical methods. Margado et al. \cite{morgado} analyzed numerical schemes for factional-order delay differential equations. Bhalekar et al. used a predictor-corrector scheme for solving nonlinear delay differential equations of fractional-order. One of numerical methods for fractional delay differential equations was provided by Wang \cite{wang}. 

The concept of \textbf{controllability} is a qualitative property of dynamical systems in control theory and an essential structure of applications-oriented mathematics. In above all, control theory characterize a key role in both deterministic and stochastic control systems. In the last few years, controllability problems for different types of linear and nonlinear  differential equations in finite and infinite dimensional spaces have been established in many publications \cite{zhou,nirmala,rajendran, mabel}. Sakthivel et al. \cite{sakthivel} described a new set of sufficient conditions for approximate controllability of nonlinear fractional stochastic evolution equation in Hilbert spaces using some techniques and methods adopted from deterministic control problems. Rajendran et al.\cite{rajendran}
obtained the sufficient conditions for complete controllability of stochastic fractional neutral systems with Wiener and L\`{e}vy noise. Meanwhile, Rajendran et al.\cite{mabel} studied the controllability of linear and nonlinear stochastic fractional systems with bounded operator having distributed delay in control. For linear case, necessary and sufficient conditions are obtained. Moreover, nonlinear system corresponding to linear system was shown under the sufficient conditions by using Banach contraction mapping principle in \cite{mabel}. For more recent research collaborations, relative controllability of systems with pure delay, control of oscillating systems with a single delay and controllability of nonlinear fractional delay dynamical systems with prescribed controls one can refer reader to study \cite{liang,khusainov-shuklin,xiao, diblik, mengmeng}. 

Very recently, the authors in \cite{arzu} established new results on the existence and uniqueness of mild solutions to stochastic neutral differential equations involving Caputo fractional time derivative operator and derived a stochastic version of variation of constants formula for Caputo fractional-order differential systems. On the other hand, a new representation of a solution to linear homogeneous fractional differential equations with a constant delay  using the Laplace integral transform and variation of constants formula via a newly defined delayed Mittag-Leffler type matrix function was introduced in terms of a three-parameter Mittag-Leffler function in \cite{ismail}. Furthermore, our current work is motivated by above ongoing studies conducted in \cite{arzu} and \cite{ismail}, in which certain important calculations involving together delay and stochastic parts are discovered. Although, we can not apply Laplace transform for stochastic part of fractional delay differential equations system, a variation of constants formula lead us to define solution of fractional stochastic delay differential equations by following some results from \cite{ismail} and to apply existence and uniqueness results to a class of stochastic fractional delay differential equations through Banach contraction mapping principle under similar concepts in \cite{arzu}. Therefore, the main point in this paper is to find explicit solution representation of fractional stochastic  delay differential equation by following the work of \cite{ismail} to study complete controllability problems for linear and nonlinear cases.

The paper includes significant updates in the theory of stochastic fractional delay differential equations and is organized as follows. Section \ref{sec:prel} is a preparatory section where we recall main results from fractional calculus and prove the powerful lemma which is used throughout the main results. Section \ref{sec: Mittag} is devoted to present an explicit solutions in terms of three-parameter Mittag-Leffler functions for homogeneous and nonhomogeneous linear fractional delay dynamical systems involving Caputo fractional derivative by using the method of variation of constants. In Section \ref{sec: main}, we prove the existence and uniqueness of the mild solution to \eqref{Problem1} with Lipschitz conditions under the Banach contraction mapping principle through the appropriate weighted maximum norm. To do so, we derive a stochastic analogue of fractional delay differential equations via newly defined delayed Mittag-Leffler type matrix function in Section \ref{sec: Mittag}, and we show the coincidence between integral equation and mild solution of \eqref{Problem1}. Section \ref{sec:control} is devoted to investigate complete controllability results for linear and nonlinear fractional stochastic delay differential equations system with Wiener noise. In Section \ref{sec:discus}, we provide an outline for our main contributions and show some open problems in the same vein of this research work.

\section{Preliminaries} \label{sec:prel}

We assume a filtered probability space $(\Omega, \mathscr{F}, \mathbb{F}_{T},\textbf{P})$ for $T>0$, with some filtration $\mathbb{F}_{T} \coloneqq \left\lbrace \mathscr{F}_{t}\right\rbrace_{t \in [0, T]}$  satisfying usual conditions, namely it is increasing and right-continuous while $\mathscr{F}_{0}$ consists of all \textbf{P}-null sets. $H^{2}([0,T], \mathbb{R}^{n})$ denote the space of all $\mathscr{F}_{T}$-measurable processes $\xi$ satisfying 
\begin{equation*}
\| \xi\|^{2}_{H^{2}} \coloneqq \sup_{t \in [0,T]} \textbf{E}\|\xi(t)\|^{2}< \infty,
\end{equation*}
where $\textbf{E}$ denotes expectation with respect to probability measure $\textbf{P}$. \\
 Let $\mathbb{R}^{n}$ be endowed with the standard Euclidean norm and $U_{ad}\coloneqq L^{\mathscr{F}}_{2}([0,T], \mathbb{R}^{n})$ be a control set.\\
Now we recall an essential structure of fractional calculus (for the more salient details on the matter, see \cite{kilbas}-\cite{samko}.

\begin{defn} \cite{diethelm}
	The Riemann-Liouville integral operator of fractional order $0<\alpha <1$ is defined by
	\begin{equation}
	(\prescript{}{}I^{\alpha}_{0^{+}}f)(t)=\frac{1}{\Gamma(\alpha)}\int_0^t(t-s)^{\alpha-1}f(s)\,\mathrm{d}s \\,\quad \text{for} \quad t>0,
	\end{equation}
	where $\Gamma:(\,0,\infty)\ \to \mathbb{R}$ is the well-known Euler's Gamma function defined as 
	\begin{equation*}
	\Gamma(\alpha) \coloneqq \int_{0}^{\infty}\tau^{\alpha-1}\exp(-\tau)\mathrm{d}\tau,
	\end{equation*} 
\end{defn}
\begin{defn} \label{beta} \cite{rainville}
	The Beta function is defined by  the definite integral:
	\begin{equation}
	\mathbb{B}(a,b)=\int_{0}^{1}\tau^{a-1}(1-\tau)^{b-1}\mathrm{d}\tau , \quad  \text{for} \quad a>0, b>0. 
	\end{equation}
\end{defn}
Also, the relation between Gamma and Beta function are as follows:
\begin{equation*}
\mathbb{B}(a,b)=\frac{\Gamma(a)\Gamma(b)}{\Gamma(a+b)},\quad \quad \text{for} \quad  a>0, b>0. 
\end{equation*}
\begin{defn} \cite{diethelm}
	The Riemann-Liouville fractional derivative of order $0< \alpha <1$ for a function \linebreak $f: [0, \infty) \to \mathbb{R}^{n}$ is defined by 
	\begin{equation}
	(\prescript{}{}D^{\alpha}_{0^{+}}f)(t)=\frac{1}{\Gamma(1-\alpha)}\frac{d}{dt}\int_0^t(t-s)^{-\alpha}f(s)\,\mathrm{d}s,\quad \text{for} \quad t>0.
	\end{equation}
\end{defn}
\begin{defn} \cite{kilbas} 
	The Caputo fractional derivative  of  order $0<\alpha<1$ for a function $f:[0,\infty)\to \mathbb{R}^{n}$ is defined by
	\begin{equation}
	(\prescript{C}{}D^{\alpha}_{0^{+}}f)(t)=\frac{1}{\Gamma(1-\alpha)}\int_0^t(t-s)^{-\alpha}f'(s)\,\mathrm{d}s \\,\quad \text{for} \quad t>0,
	\end{equation}
	in particular,
	\begin{equation}
	(\prescript{}{}I^{\alpha}_{0^{+}}\prescript{C}{}D^{\alpha}_{0^{+}})f(t)=f(t)-f(0).
	\end{equation}
\end{defn}

The classical matrix \textbf{Mittag-Leffler function}, defined as :
\begin{equation}
E_{\alpha}(At^{\alpha})= \sum_{k=0}^{\infty}A^{k}\frac{t^{k \alpha}}{\Gamma(k \alpha +1)}, \quad  \alpha \in \mathbb{R}_{+}, t \in \mathbb{R}, A\in\mathbb{R}^{n\times n},
\end{equation}
has been extended and generalized in different ways, with functions denoted by "two-parameter" $E_{\alpha,\beta}(At^{\alpha})$ and "three-parameter" $E^{\delta}_{\alpha,\beta}(At^{\alpha})$ matrix Mittag-Leffler functions. 

\begin{equation}
E_{\alpha,\beta}(At^{\alpha})= \sum_{k=0}^{\infty}A^{k}\frac{t^{k \alpha}}{\Gamma(k \alpha +\beta)}, \quad  \alpha,\beta \in \mathbb{R}_{+}, t \in \mathbb{R}, A\in\mathbb{R}^{n\times n},
\end{equation}

	\begin{equation}
	E^{\delta}_{\alpha,\beta}(At^{\alpha})=\sum_{k=0}^{\infty}A^{k}\frac{(\delta)_{k}}{\Gamma(k \alpha +\beta)}\frac{t^{k\alpha}}{k!}, \alpha,\beta,\delta \in\mathbb{R_{+}}, t\in\mathbb{R},A\in\mathbb{R}^{n\times n},
	\end{equation}
	where $(\delta)_{k}$ is the Pochhammer symbol \cite{rainville}. 

The following lemma is powerful tool to obtain certain estimations in the main results of the theory.

\begin{lem} \label{ineqMit}
	For all $\gamma, t>0$, we have 
	\begin{equation*}
	\frac{\gamma}{\Gamma(2\alpha-1)}\int_{0}^{t}(t-s)^{2\alpha-2}E_{2\alpha-1}(\gamma s^{2\alpha-1})ds \leq E_{2\alpha-1}(\gamma t^{2\alpha-1}).
	\end{equation*}
\end{lem}
\begin{proof}
\begin{align*}
	\frac{\gamma}{\Gamma(2\alpha-1)}\int_{0}^{t}(t-s)^{2\alpha-2}E_{2\alpha-1}(\gamma s^{2\alpha-1})ds
	&= \frac{\gamma}{\Gamma(2\alpha-1)}
	\sum_{i=0}^{\infty}\frac{\gamma^{i}}{\Gamma(i(2\alpha-1)+1)}\int_{0}^{t}(t-s)^{2\alpha-2}s^{i(2\alpha-1)}ds \\
	&=\sum_{i=0}^{\infty}\frac{\gamma^{i+1}t^{(i+1)(2\alpha-1)}}{\Gamma(2\alpha-1)\Gamma(i(2\alpha-1)+1)}\mathbb{B}(2\alpha-1, i(2\alpha-1)+1) \\
	&= \sum_{i=0}^{\infty} \frac{\gamma^{i+1}t^{(i+1)(2\alpha-1)}}{\Gamma((i+1)(2\alpha-1)+1)}
	=\sum_{i=1}^{\infty}\frac{\gamma^{i}t^{i(2\alpha-1)}}{\Gamma(i(2\alpha-1)+1)}\\
	&=E_{2\alpha -1}(\gamma t^{2\alpha -1})-1 \leq E_{2\alpha -1}(\gamma t^{2\alpha -1}),
	\end{align*}
	where $\mathbb{B}$ is a Beta function.
\end{proof}

\section{Deterministic analogue of a fractional stochastic time-delay system} \label{sec: Mittag}

We consider linear homogeneous fractional time-delay system with a single constant delay:

\begin{equation} \label{11}
\begin{cases}
(\prescript{C}{}D^{\alpha}_{0^{+}}x)(t)= A x(t)+Bx(t-h), \quad x(t) \in \mathbb{R}^{n}, \quad t \in (0, T], \quad h>0, \\
x(t)=\phi(t), \quad -h\leq t \leq 0.
\end{cases}
\end{equation}
where $A,B \in \mathbb{R}^{n \times n}$ are permutable matrices, i.e., $AB=BA$ and $\phi : [-h, 0] \to \mathbb{R}^{n}$ is an arbitrary differentiable vector function, i.e.,$\phi \in C^{1}([-h,0], \mathbb{R}^{n})$, and  $T=nh$ for a fixed natural number $n$.\\

The following definitions and theorems, which are provided in \cite{ismail}, are important results to derive stochastic version of variation of constants formula for fractional delay differential equations in the next section.
\begin{defn} \label{df1}
	Delayed classical Mittag-Leffler type matrix function of three parameters $\mathscr{E}^{A,B}_{h,\alpha}:\mathbb{R}\to \mathbb{R}^{n\times n}$ is defined by
	
	\begin{equation} \label{delayMittagAB}
	\mathscr{E}^{A,B}_{h,\alpha}(t)=
	\begin{cases}
	\Theta, \quad -\infty <t< -h, \\
	I, \quad -h\leq t\leq 0, \\
	I+ t^{\alpha}E^{1}_{\alpha,\alpha+1}(At^{\alpha})(A+B)
	+(t-h)^{2\alpha}BE^{2}_{\alpha,2\alpha+1}(A(t-h)^{\alpha})(A+B)\\
	+\cdots+ (t-(n-1)h)^{n\alpha}B^{n-1}E^{n}_{\alpha,n\alpha+1}(A(t-(n-1)h)^{\alpha})(A+B), (n-1)h <t\leq nh,
	\end{cases}
	\end{equation}
	where $\Theta \in \mathbb{R}^{n \times n}$ and $I \in \mathbb{R}^{n \times n}$ denote the zero and identity matrices, respectively.
\end{defn}
We denote the matrix function by $X(t)$ that  is a solution of the matrix differential equation:
\begin{equation} \label{2.2}
(\prescript{C}{}D^{\alpha}_{0^{+}}X)(t)= AX(t)+ BX(t-h), \quad t > 0,
\end{equation}
with the unit initial conditions:
\begin{equation} \label{333}
X(t)=
\begin{cases}
 I, \quad -h \leq t \leq 0,\\ \Theta, \quad t < -h.
\end{cases}
\end{equation}
\begin{thm} \label{thm666}
	The solution of equation \eqref{2.2} satisfying initial conditions \eqref{333} has the form:

	\begin{align*}
	&X(t)= I+ \sum_{k=0}^{n-1} \varphi_{k}(t), \quad if \quad (n-1)h < t \leq nh, \\
	&\varphi_{k}(t)= (t-kh)^{(k+1)\alpha}\sum_{i=0}^{\infty} \begin{pmatrix}
	k+i \\ i \end{pmatrix} A^{i}\frac{(t-kh)^{i\alpha}}{\Gamma((k+i+1)\alpha+1)}B^{k}(A+B)\\
	&=(t-kh)^{(k+1)\alpha}E^{k+1}_{\alpha,(k+1)\alpha+1}(A(t-kh)^{\alpha}) B^{k}(A+B), \quad t> kh,
	\end{align*}
	where 
	\begin{equation*}
\begin{pmatrix}	k+i \\ i \end{pmatrix}	= \frac{(k+1)_{i}}{i!}.
	\end{equation*}
\end{thm}
$\mathscr{E}^{A,B}_{h,\alpha}(t)$ which is stated in the following theorem coincides with $X(t)$ that is introduced in Theorem \ref{thm666}.
\begin{thm} \label{th3.1}
	For a delayed Mittag-Leffler type matrix function of three-parameters $\mathscr{E}^{A,B}_{h,\alpha}:\mathbb{R}\to \mathbb{R}^{n\times n}$, one has 
	\begin{equation} \label{D2}
	(\prescript{C}{}D^{\alpha}_{0^{+}}\mathscr{E}^{A,B}_{h,\alpha})(t)= A\mathscr{E}^{A,B}_{h,\alpha}(t) + B\mathscr{E}^{A,B}_{h,\alpha}(t-h),
	\end{equation}
	i.e., $\mathscr{E}^{A,B}_{h,\alpha}$ is a solution of matrix fractional differential equation:
	\begin{equation} \label{matX}
	(\prescript{C}{}D^{\alpha}_{0^{+}}X)(t)= AX(t) + BX(t-h)
	\end{equation}
	which satisfies unit initial conditions $\mathscr{E}^{A,B}_{h,\alpha}(t)=I, -h \leq t \leq 0$ and \quad  $\mathscr{E}^{A,B}_{h,\alpha}(t)=\Theta, t< -h$.
\end{thm}

\begin{thm}
	A solution $x\in C([-h,T], \mathbb{R}^{n})$ of \eqref{11} can be represented by the following formula:
	\begin{equation} \label{detsol}
	x(t)=\mathscr{E}^{A,B}_{h,\alpha}(t)\phi(-h)+ \int_{-h}^{0}\mathscr{E}^{A,B}_{h,\alpha}(t-h-r)\phi^{\prime}(r)dr.
	\end{equation}
\end{thm}

Now, we consider nonhomogeneous case of fractional delay differential system according to \eqref{11} in the following form:
\begin{equation} \label{1.1}
\begin{cases}
(\prescript{C}{}D^{\alpha}_{0^{+}}x)(t)=Ax(t)+Bx(t-h)+f(t), t\in(0,T], \quad h > 0,\\ x(t)=\varphi(t), -h \leq t \leq 0,
\end{cases}
\end{equation}
where  $f \in  C([0,T], \mathbb{R}^{n})$ is a nonlinear perturbation.
\begin{defn}
	Delayed perturbation of three-parameter Mittag-Leffler type matrix function  \linebreak $\mathscr{E}^{A,B}_{h,\alpha,\beta}:\mathbb{R}\to \mathbb{R}^{n\times n}$ is defined by
	\begin{equation} \label{delMitABtau}
	\mathscr{E}^{A,B}_{h,\alpha,\beta}(t)=
	\begin{cases*}
	\Theta, \quad -\infty <t\leq -h, \\
	(t+h)^{\beta-1}E^{1}_{\alpha,\beta}(A(t+h)^{\alpha}), \quad -h<t\leq 0, \\
	(t+h)^{\beta-1}E^{1}_{\alpha,\beta}(A(t+h)^{\alpha})+ t^{\alpha+\beta-1}BE^{2}_{\alpha,\alpha+\beta}(At^{\alpha})\\
	+\cdots+ (t-(n-2)h)^{(n-1)\alpha+\beta-1}B^{n-1}E^{n}_{\alpha,(n-1)\alpha+\beta}(A(t-(n-2)h)^{\alpha})\\
	+(t-(n-1)h)^{n\alpha+\beta-1}B^{n}E^{n+1}_{\alpha,n\alpha+\beta}(A(t-(n-1)h)^{\alpha}), \quad (n-1)h <t\leq nh,
	\end{cases*}
	\end{equation}
	where $\Theta \in \mathbb{R}^{n \times n}$ and $I \in \mathbb{R}^{n \times n}$ denote the zero and identity matrices, respectively.
\end{defn}

	\begin{thm}
		A solution $\tilde{x}\in C([-h, T], \mathbb{R}^{n})$ of \eqref{1.1} satisfying zero initial condition $x(t)\equiv 0, t \in [-h,0]$ has the following form
		\begin{equation} \label{tildex}
		\tilde{x}(t)= \int_{0}^{t}\mathscr{E}^{A,B}_{h,\alpha,\alpha}(t-h-r)f(r)dr, \quad  t > 0.
		\end{equation}
	\end{thm}
	
	The following corollary present the construction of formula of solutions to \eqref{1.1} . The proof is straightway, so we pass over it here.
	
	\begin{thm} \label{5.2}
		The solution $x \in C([-h,T], \mathbb{R}^{n})$ of \eqref{1.1} has a form
		\begin{equation*}
		x(t)=\mathscr{E}^{A,B}_{h,\alpha}(t)\phi(-h)+ \int_{-h}^{0}\mathscr{E}^{A,B}_{h,\alpha}(t-h-r)\phi^{\prime}(r)dr+ \int_{0}^{t}\mathscr{E}^{A,B}_{h,\alpha,\alpha}(t-h-r)f(r)dr.
		\end{equation*}
		
	\end{thm}	
		The following lemma is a necessary tool on certain calculations in the main results of the theory. 

\begin{lem} \label{estimate}
	For any permutable matrices $A, B \in \mathbb{R}^{n\times n}$, we have, 
	\begin{equation}
	\|\mathscr{E}^{A,B}_{h,\alpha,\beta}(t)\|\leq\sum_{k=0}^{n}t^{k\alpha+\beta-1}\|B\|^{k}E^{k+1}_{\alpha,k\alpha+\beta}(\|A\|t^{\alpha}), (n-1)h<t\leq nh, \quad n\in\mathbb{N}.
	\end{equation}
\end{lem}
\begin{proof}
	We estimate $\mathscr{E}^{A,B}_{h,\alpha,\beta}$ as follows:
	\begin{align*}
	\|\mathscr{E}^{A,B}_{h,\alpha,\beta}(t)\| &=\|\sum_{k=0}^{n}(t-(k-1)h)^{k\alpha+\beta-1}B^{k}E^{k+1}_{\alpha,k\alpha+\beta}(A(t-(k-1)h)^{\alpha})\|\\
	&\leq\|t^{\beta-1}E^{1}_{\alpha,\beta}(At^{\alpha})
	+t^{\alpha+\beta-1}BE^{2}_{\alpha,\alpha+\beta}(At^{\alpha})
	+\dots+t^{n\alpha+\beta-1}B^{n}E^{n+1}_{\alpha,n\alpha+\beta}(At^{\alpha})\|\\
	&\leq t^{\beta-1}E^{1}_{\alpha,\beta}(\|A\|t^{\alpha})
	+t^{\alpha+\beta-1}\|B\|E^{2}_{\alpha,\alpha+\beta}(\|A\|t^{\alpha})
	+\dots+t^{n\alpha+\beta-1}\|B\|^{n}E^{n+1}_{\alpha,n\alpha+\beta}(\|A\|t^{\alpha})\\
	&=\sum_{k=0}^{n}t^{k\alpha+\beta-1}\|B\|^{k}E^{k+1}_{\alpha,k\alpha+\beta}(\|A\|t^{\alpha}), \quad \text{for any} \quad (n-1)h < t \leq nh, \quad n \in \mathbb{N}.
	\end{align*}
\end{proof}

\section{Main results} \label{sec: main}
 In this section, we derive stochastic version of variation of constants formula for fractional delay differential equations via newly defined delayed analogue of  three-parameter Mittag-Leffler type matrix function in Section 3. Then, we study the global existence and uniqueness of a mild solution to a stochastic fractional delay differential equations system. The main part here is to use the weighted maximum norm and to prove a coincidence between the notion of the integral equation of \eqref{integral equation} and mild solution of \eqref{3.3}. 
 
 Consider a Caputo type fractional stochastic delay differential equations system of order $\alpha \in (\frac{1}{2}, 1)$ on a bounded interval $[0,T]$ of the following form:
 \begin{equation} \label{Problem1}
 \begin{cases}
 (\prescript{C}{}D^{\alpha}_{0^{+}} x)(t)=Ax(t)+Bx(t-h)+\Delta(t)\frac{dW(t)}{dt}, \quad h >0\\
 x(t)=\phi(t), \quad t \in [-h, 0],
 \end{cases}
 \end{equation}  
where $A, B \in \mathbb{R}^{n\times n}$ are permutable matrices, $x(t) \in \mathbb{R}^{n}$, $\Delta \in C([0,T], \mathbb{R}^{n})$, $(W(t))_{t \geq 0}$ is a standard Brownian motion on a complete probability space $(\Omega, \mathbb{F}_{T},\textbf{P})$ and the initial condition  $\phi: [-h,0] \to \mathbb{R}^{n}$ is an arbitrary differentiable vector function, i.e.,   $\phi (\cdot) \in C^{1}([-h,0], \mathbb{R}^{n})$ . \\
The corresponding nonlinear system to \eqref{Problem1}:

\begin{equation} \label{Problem2}
\begin{cases}
(\prescript{C}{}D^{\alpha}_{0^{+}}x)(t)=Ax(t)+Bx(t-h)+\Delta(t,x(t))\frac{dW(t)}{dt}, \quad h >0\\
x(t)=\phi(t), \quad t \in [-h, 0],
\end{cases}
\end{equation}
where  $\Delta:[0,T]\times \mathbb{R}^{n} \to \mathbb{R}^{n}$ is measurable and bounded function satisfying following conditions: \\ 
(A1)  There exists $L_{\Delta}>0$ such that for all $x,y \in \mathbb{R}^{n}, t \in [0,T]$ ,

\begin{equation*} 
\|\Delta(t.x)-\Delta(t,y)\|\leq L_{\Delta}\|x-y\|,
\end{equation*}
(A2) $ \enspace ess\sup_{t\in [0,T]}\|\Delta(t,0)\|< \infty $.

 \begin{defn}
 	A stochastic process $ \left\lbrace x(t), t \in [0,T]\right\rbrace$ is called  a mild solution of \eqref{Problem2} if 
 	\begin{itemize}
 		\item $x(t)$ is adapted to $\left\lbrace \mathscr{F}_{t}\right\rbrace _{t \geq 0}$ with 
 		$ \int_{0}^{T} \|x(t)\|^{2}_{H^{2}}dt < \infty$ a.s.;
 		\item $x$ $\in H^{2}([0,T], \mathbb{R}^{n})$ has continuous path on $[0,T]$ a.s. and for each $t\in [0,T]$, $x(t)$ satisfies the following integral equation:
 	\end{itemize}	
 	\begin{align}\label{integral equation}
 	x(t)= \phi(0)+\frac{1}{\Gamma(\alpha)}  \int_{0}^{t}(t-r)^{\alpha-1}\left[ Ax(r)+Bx(r-h)\right] dr +\frac{1}{\Gamma(\alpha)}\int_{0}^{t}(t-r)^{\alpha-1}\Delta(r, x(r))dW(r).
 	\end{align}	
 \end{defn}
 
 \begin{thm}[A variation of constants formula for Caputo fractional stochastic delay differential equations] \label{theorem1}
 	The unique mild solution $x \in H^{2}([0,T], \mathbb{R}^{n})$ of \eqref{Problem2} with initial condition $x(t)=\phi(t)$, $ t \in [-h,0]$, has the following form: 
 	\begin{equation} \label{3.3}
 	x(t)=\mathscr{E}^{A,B}_{h,\alpha}(t)\phi(-h)+ \int_{-h}^{0}\mathscr{E}^{A,B}_{h,\alpha}(t-h-r)\phi^{\prime}(r)dr+ \int_{0}^{t}\mathscr{E}^{A,B}_{h,\alpha,\alpha}(t-h-r)\Delta(r, x(r))dW(r).
 	\end{equation}
 \end{thm}
 
\begin{thm}  \label{theorem}
	Suppose that hypotheses (A1) and (A2) hold. Then, the unique mild solution $y(t)$ of (\ref{Problem2}) satisfying $y(t)=\phi(t)$, $t \in [-h,0]$ can be expressed in the following form:

	\begin{equation} \label{y}
	y(t)=\mathscr{E}^{A,B}_{h,\alpha}(t)\phi(-h)+ \int_{-h}^{0}\mathscr{E}^{A,B}_{h,\alpha}(t-h-r)\phi^{\prime}(r)dr+ \int_{0}^{t}\mathscr{E}^{A,B}_{h,\alpha,\alpha}(t-h-r)\Delta(r,y(r))dW(r).
	\end{equation}
	
\end{thm}

\begin{proof} 
	To introduce a fixed point theorem associated with \eqref{Problem2}, we define the operator $\mathscr{T}: H^{2}([0,T], \mathbb{R}^{n}) \to H^{2}([0,T], \mathbb{R}^{n})$ by
	\begin{equation}
	(\mathscr{T}y)(t)= \mathscr{E}^{A,B}_{h,\alpha}(t)\phi(-h)
    +\int_{-h}^{0}\mathscr{E}^{A,B}_{h,\alpha}(t-h-r)\phi^{\prime}(r)dr
    +\int_{0}^{t}\mathscr{E}^{A,B}_{h,\alpha, \alpha}(t-h-r)\Delta(r,y(r))dW(r).	
	\end{equation}
It follows that $\mathscr{T}$ is well-defined. Let  $H^{2}([0,T], \mathbb{R}^{n})$ be endowed  with the weighted maximum norm $ \| \cdot \|_{\gamma}$, where $\gamma >0$, defined as
\begin{equation} \label{maxnorm}
\| \xi\|^{2}_{\gamma} \coloneqq \sup_{t \in [0,T]}\frac{\textbf{E}\| \xi (t)\|^{2}}{E_{2\alpha -1}(\gamma t^{2\alpha-1})}, \\  \quad \text{for all} \quad \xi \in H^{2}([0,T], \mathbb{R}^{n}).
\end{equation}
Apparently,  $(H^{2}([0,T], \mathbb{R}^{n}),\| \cdot \|_{H^{2}})$ is a Banach space. Since two norms $ \| \cdot \|_{H^{2}}  \quad \text{and} \quad \| \cdot \|_{\gamma}$ are equivalent,  $(H^{2}([0,T], \mathbb{R}^{n}),\| \cdot \|_{\gamma})$ is also Banach space. Therefore, it is complete. For simplicity, take $M_{k} \coloneqq \max_{t \in [0,T]} E^{k+1}_{\alpha, (k+1)\alpha}(\|A\|t^{\alpha})$, for $k=0,1,...,n$. Choose and fix a positive $\gamma$ such that 
\begin{equation} \label{cond}
\frac{L^{2}_{\Delta}\lambda_{T}}{\gamma}<1,
\end{equation} 
where $\lambda_{T} \coloneqq  \Gamma(2\alpha-1)\sum_{k=0}^{n}M^{2}_{k} \|B\|^{2k}T^{2k}$.

By assumption of (A1), definition of $M_{k}, k=0,1,...,n$, using Lemma \ref{estimate} and It\^{o}'s isometry, we have
\allowdisplaybreaks
\begin{align*}
\textbf{E}\|(\mathscr{T}x)(t)-(\mathscr{T}y)(t)\|^{2} &=\textbf{E} \|\int_{0}^{t}\mathscr{E}^{A,B}_{h,\alpha, \alpha}(t-h-r)\left[ \Delta(r,x(r))-\Delta(r,y(r))\right] dW(r)\|^{2} \\
&=\textbf{E}\int_{0}^{t}\|\mathscr{E}^{A,B}_{h,\alpha,\alpha}(t-h-r)\left[ \Delta(r,x(r))-\Delta(r,y(r))\right]\|^{2}dr\\
&\hspace{-3cm}\leq L^{2}_{\Delta}\int_{0}^{t}\|\sum_{k=0}^{n}B^{k}(t-h-r)^{(k+1)\alpha-1}E^{k+1}_{\alpha, (k+1)\alpha}(A(t-h-r)^{\alpha})\|^{2}\textbf{E}\|x(r)-y(r)\|^{2}dr\\
&\leq L^{2}_{\Delta}\int_{0}^{t}\|\sum_{k=0}^{n}B^{k}(t-r)^{(k+1)\alpha-1}E^{k+1}_{\alpha, (k+1)\alpha}(A(t-r)^{\alpha})\|^{2}\textbf{E}\|x(r)-y(r)\|^{2}dr\\
&\leq L^{2}_{\Delta}\sum_{k=0}^{n}\|B\|^{2k}t^{2k}\|E^{k+1}_{\alpha, (k+1)\alpha}(At^{\alpha})\|^{2}\int_{0}^{t}(t-r)^{2\alpha-2}\textbf{E}\|x(r)-y(r)\|^{2}dr\\
&\hspace{-3cm}\leq L^{2}_{\Delta}\sum_{k=0}^{n}\|B\|^{2k}t^{2k}\Big(E^{k+1}_{\alpha, (k+1)\alpha}(\|A\|t^{\alpha})\Big)^{2}\int_{0}^{t}(t-r)^{2\alpha-2}\textbf{E}\|x(r)-y(r)\|^{2}dr\\
&\leq L^{2}_{\Delta}\sum_{k=0}^{n}M^{2}_{k}\|B\|^{2k}T^{2k}\int_{0}^{t}(t-r)^{2\alpha-2}\textbf{E}\|x(r)-y(r)\|^{2}dr.
\end{align*}
Hence, by definition of $\|\cdot \|_\gamma$ and Lemma \ref{ineqMit}, we derive the following expression:
\begin{align*}
\frac{\textbf{E}\|(\mathscr{T}x)(t)-(\mathscr{T}y)(t)\|^{2}}{E_{2\alpha-1}(\gamma t^{2\alpha-1})} &\leq  
 L^{2}_{\Delta} \sum_{k=0}^{n}M^{2}_{k}\|B\|^{2k}T^{2k}\frac{1}{E_{2\alpha-1}(\gamma t^{2\alpha-1})} \\
&\times \int_{0}^{t}(t-r)^{2\alpha-2}\frac{E_{2\alpha-1}(\gamma r^{2\alpha-1})}{E_{2\alpha-1}(\gamma r^{2\alpha-1})}\textbf{E}\|x(r)-y(r)\|^{2}dr\\
&\leq  L^{2}_{\Delta}\frac{\Gamma(2\alpha-1)}{\gamma}\sum_{k=0}^{n}M^{2}_{k}\|B\|^{2k}T^{2k}\|x-y\|^{2}_{\gamma}.
\end{align*}

Using weighted maximum norm \eqref{maxnorm}, we achieve
\begin{equation*}
\|\mathscr{T}x-\mathscr{T}y\|^{2}_{\gamma} \leq  L^{2}_{\Delta}\frac{\Gamma(2\alpha-1)}{\gamma}\sum_{k=0}^{n}M^{2}_{k}\|B\|^{2k}T^{2k}\|x-y\|^{2}_{\gamma}.
\end{equation*}
Therefore,
\begin{equation}
\|\mathscr{T}x-\mathscr{T}y\|^{2}_{\gamma} \leq  \frac{L^{2}_{\Delta}\lambda_{T}}{\gamma} \| x-y\|^{2}_{\gamma},
\end{equation}
which together with (\ref{cond}) implies that $\mathscr{T}$ is contraction on $H^{2}([0,T], \mathbb{R}^{n})$. By contraction mapping principle, $\mathscr{T}$ has a unique fixed point and the proof is complete.
\end{proof}

  Using the martingale representation theorem for any function $f \in \mathbb{L}^{2}(\Omega, \mathscr{F}_{T}, \mathbb{R}^{n})$, there exists a unique adapted process $\Theta \in H^{2}([0,T],\mathbb{R}^{n}$ such that
\begin{equation*}
f=\textbf{E}f+\int_{0}^{T}\Theta(r)dW(r).
\end{equation*}
It is clear that
\begin{equation*}
f= \sum_{k=1}^{n}f_{k}e_{k}, \quad f_{k}=\textbf{E}f_{k}+\int_{0}^{T}\theta_{k}(r)dW(r), \quad f_{k} \in \mathbb{L}^{2}(\Omega, \mathscr{F}_{T}, \mathbb{R}).
\end{equation*}
For \eqref{integral equation} and \eqref{y}, it is sufficient to show that 
\begin{equation} \label{phi and psi}
x(t)=y(t).
\end{equation}
To show (\ref{phi and psi}), it is enough to prove that for any $f \in \mathbb{L}^{2}(\Omega, \mathscr{F}_{T}, \mathbb{R}^{n})$, 
\begin{equation*}
\textbf{E}\langle x(t),f \rangle = \textbf{E}\langle y(t), f \rangle.
\end{equation*}
In other words,
\begin{equation*}
\textbf{E}\langle x(t)-y(t),f \rangle = \sum_{k=1}^{n}\textbf{E}\langle (x(t)-y(t))f_{k}, e_{k} \rangle.
\end{equation*}
It follows that
\begin{equation*}
|\textbf{E}\langle x(t)-y(t),f \rangle|^{2} \leq \Big(\sum_{k=1}^{n}\|\textbf{E}(x(t)-y(t))f_{k}\|\Big)^{2}
\leq n \sum_{k=1}^{n}\|\textbf{E}(x(t)-y(t))f_{k}\|^{2}.
\end{equation*}
Before estimating $|\textbf{E}\langle x(t)-y(t),f \rangle|$, define the following functions:
\begin{equation*}
\chi_{k}(t)= \textbf{E}x(t)f_{k}, \quad \tilde{\chi}_{k}(t)= \textbf{E}y(t)f_{k}, \quad
\chi_{k}(t-h)= \textbf{E}x(t-h)f_{k}, \quad \tilde{\chi}_{k}(t-h)= \textbf{E}y(t-h)f_{k}.
\end{equation*}
\begin{remark} \label{remark1}
	Since $x(t), y(t) \in  H^{2}([0,T], \mathbb{R}^{n})$ , the functions $\chi_{k}(t), \chi_{k}(t-h), \tilde{\chi}_{k}(t), \tilde{\chi}_{k}(t-h)$ are measurable and bounded  on $[0,T]$. 
\end{remark}
 
\begin{lem}
For all $t \in [0,T]$ and $c \in \mathbb{R}^{n}$, the following statements hold:

\begin{equation} \label{chi}
\chi_{k}(t)=c\textbf{E}\phi(0)
+\frac{1}{\Gamma(\alpha)}  \int_{0}^{t}(t-r)^{\alpha-1}\left[ A\chi_{k}(r)+B\chi_{k}(r-h)+\textbf{E}\theta_{k}(r)\Delta(r,x(r))\right]dr,
\end{equation}	

\begin{equation} \label{tilde{chi}}
\tilde{\chi}_{k}(t)=c\mathscr{E}^{A,B}_{h,\alpha}(t)\phi(-h)
+\int_{-h}^{0}\mathscr{E}^{A,B}_{h,\alpha}(t-h-r)\phi^{\prime}(r)dr
+\int_{0}^{t}\mathscr{E}^{A,B}_{h,\alpha, \alpha}(t-h-r)\textbf{E}\theta_{k}(r)\Delta(r,y(r))dr.
\end{equation}
\end{lem}

\begin{proof}
	Taking product of both sides of \eqref{integral equation} with $f_{k}$  and then taking expectation of both sides give that
	
	\begin{align*}
	\chi_{k}(t)=c\textbf{E}\phi(0)&+\frac{1}{\Gamma(\alpha)}  \int_{0}^{t}(t-r)^{\alpha-1}\left[ A\chi_{k}(r)
	+B\chi_{k}(r-h)\right] dr \\ 
	&+\frac{1}{\Gamma(\alpha)}\textbf{E}\Big( \int_{0}^{t}(t-r)^{\alpha-1}\Delta(r,x(r))dW(r)\Big)\int_{0}^{T}\theta_{k}(r)dW(r).
	\end{align*}
	
	Using It\^{o}'s representation theorem, we attain 
	\begin{equation*}
	\chi_{k}(t)=c\textbf{E}\phi(0)
	+\frac{1}{\Gamma(\alpha)}  \int_{0}^{t}(t-r)^{\alpha-1}\left[ A\chi_{k}(r)+B\chi_{k}(r-h)+\textbf{E}\theta_{k}(r)\Delta(r,x(r))\right]dr
	\end{equation*}
	such that  $\chi_{k}(t)$ is a solution of the following fractional delay differential equation:
	\begin{equation}
	(\prescript{C}{}{D^{\alpha}_{0^{+}}}x)(t)= Ax(t)+Bx(t-h) +\textbf{E}\theta_{k}(t)\Delta(t,x(t)), \quad x(t)= c\textbf{E}\phi(t), t \in [-h,0].
	\end{equation}
	Then, by means of Remark \ref{remark1},  (\ref{chi}) is proved. \\

	Similarly, by taking product of both sides of \eqref{y} with $f_{k}$ and expectation of both sides yield that
	\begin{align*}
	\tilde{\chi}_{k}(t)=c\mathscr{E}^{A,B}_{h,\alpha}(t)\textbf{E}\phi(-h)
	&+ \int_{-h}^{0}\mathscr{E}^{A,B}_{h,\alpha}(t-h-r)\textbf{E}f_{k}\phi^{\prime}(r)dr \\ 
	&+\textbf{E}\Big( \int_{0}^{t}\mathscr{E}^{A,B}_{h,\alpha, \alpha}(t-h-r)\Delta(r,y(r))dW(r)\Big)\int_{0}^{T}\theta_{k}(r)dW(r).
	\end{align*}
	Then, using It\^{o}'s isometry theorem, we achieve \eqref{tilde{chi}} such that $\tilde{\chi}_{k}(t)$ is a solution of the following fractional delay differential equation:
	\begin{equation}
	(\prescript{C}{}{D^{\alpha}_{0^{+}}}y)(t)= Ay(t)+By(t-h) +\textbf{E}\theta_{k}(t)\Delta(t,y(t)), \quad y(t)= c\textbf{E}\phi(t), t \in [-h,0].
	\end{equation}
	The proof is complete.
\end{proof}
\begin{remark} \label{remark2}
	For any $f \in \mathbb{L}^{2}(\Omega, \mathscr{F}_{T}, \mathbb{R}^{n})$, we have	
	 \begin{equation} \label{remark22}
	 |\textbf{E}\langle x(t)-y(t), f\rangle|^{2} \leq n N^{2}L^{2}_{\Delta}\int_{0}^{t}\textbf{E}\|x(r)-y(r)\|^{2}dr\textbf{E}\|f\|^{2}.
	 \end{equation}
	 \end{remark}
 
 \begin{proof}
 Starting with
	 
	 \begin{align}\label{estimation}
	 |\langle x(t)-y(t),f \rangle |
	 \leq \sqrt{n\sum_{k=1}^{n}|\langle x_{k}(t)-y_{k}(t),f_{k} \rangle |^{2}} 
	 &\leq\sqrt{n\sum_{k=1}^{n}\|\textbf{E}( x_{k}(t)-y_{k}(t))f_{k}\|^{2}} \nonumber \\ 
	 &= \sqrt{n\sum_{k=1}^{n}\|\chi_{k}(t)-\tilde{\chi}_{k}(t)\|^{2}},
	 \end{align}

	estimate $\|\chi_{k}(t)-\tilde{\chi}_{k}(t)\|$ by using H\"{o}lder's inequality and taking $N \coloneqq \sup_{t \in [0,T]}\|\mathscr{E}^{A,B}_{h,\alpha, \alpha}(t)\|$:
	\begin{equation*}
    \|\chi_{k}(t)-\tilde{\chi}_{k}(t)\| \leq N L_{\Delta}\Big(\int_{0}^{t}\textbf{E}\|\theta_{k}(r)\|^{2}\Big)^{\frac{1}{2}}\Big(\int_{0}^{t}\textbf{E}\|x(r)-y(r)\|^{2}dr\Big)^{\frac{1}{2}}.
	\end{equation*}
Plugging above inequality into \eqref{estimation}, we deduce the desired result \eqref{remark22}. The proof is complete.
\end{proof}

\textbf{Proof of Theorem \ref{theorem1}}. Let $T^{*}= \inf \left\lbrace t \in [0,T]; x(t) \neq y(t) \right\rbrace$. Then it is sufficient to show that $T^{*}=T$. \\
Suppose the contrary : $T^{*}< T$. Choose and fix an arbitrary $\delta >0$ satisfying the following expression:
  \begin{equation} \label{star}
  n N^{2}L^{2}_{\Delta} \delta<1.
  \end{equation}
   
To lead contradiction, we show that $x(t) = y(t)$ for all $t \in [T^{*},T^{*}+\delta]$. Using Ito's representation, there exists a unique $f \in H^{2}([0,T], \mathbb{R}^{n})$ such that  $x(t) - y(t)= f$. Therefore, we have 
   \begin{equation*}
   \textbf{E}\|x(t) - y(t)\|^{2}= \textbf{E}\|f\|^{2}.
   \end{equation*}
   Using Remark \ref{remark2}, we attain 
   \begin{equation*}
   \textbf{E}\|x(t) - y(t)\|^{2} \leq n N^{2} L^{2}_{\Delta} \int_{T^{*}}^{t} \textbf{E}\| x(r) - y(r)\|^{2} dr.
   \end{equation*}
 As a consequence,
   \begin{equation*}
   \sup_{t \in [T^{*},T^{*}+\delta]}\textbf{E}\|x(t) - y(t)\|^{2} 
   \leq n N^{2} L^{2}_{\Delta} \delta  \sup_{t \in [T^{*},T^{*}+\delta]}\textbf{E}\|x(t) - y(t)\|^{2}.
  \end{equation*}
   
   By selecting $\delta$ as in \eqref{star}, we have  $\sup_{t \in [T^{*},T^{*}+\delta]}\textbf{E}\|x(t) - y(t)\|^{2}=0$. This leads to a contradiction and the proof is complete. $\blacksquare$

\section{Controllability results} \label{sec:control}
   We shall prove controllability results for linear and nonlinear stochastic delay dynamical systems with fractional-order. First, we present necessary and sufficient conditions for complete controllability of linear fractional stochastic delay system through perturbed controllability matrix and rank condition using the rank correlation of Cayley-Hamilton theorem. Thereafter, sufficient conditions are derived for complete controllability of nonlinear fractional stochastic delay differential equations system using Banach's fixed point theorem.
   
 \begin{defn}
	The fractional stochastic delay system is said to be complete controllable on $[0,T]$, if for every initial condition $\phi(t)$ and $x_{1} \in \mathbb{L}^{2}(\Omega, \mathscr{F}_{T}, \mathbb{R}^{n})$, there exists a control $u \in U_{ad}$, such that the solution $x(t)$  satisfies $x(T)=x_{1}$.	 
\end{defn}

   \subsection{Linear case}
   
   Consider the linear fractional stochastic delay dynamical system on $[0,T]$ of the form :
   
   \begin{equation} \label{Problem3}
   \begin{cases}
   (\prescript{C}{}D^{\alpha}_{0^{+}}x)(t)=Ax(t)+Bx(t-h)+Cu(t)+\Delta(t)\frac{dW(t)}{dt}, \quad h >0,\\
   x(t)=\phi(t), \quad t \in [-h, 0],
   \end{cases}
   \end{equation} 
where $A, B \in \mathbb{R}^{n\times n}$ are permutable matrices, and $C \in \mathbb{R}^{n\times m}$ with $ n>m$, the transposes of $A, B$ and $C$ are denoted by $A^{*}, B^{*}$ and $C^{*}$, respectively; $\Delta \in C([0,T], \mathbb{R}^{n\times n})$, $x(t) \in \mathbb{L}^{2}(\Omega, \mathscr{F}_{T}, \mathbb{R}^{n})$, control function $u \in U_{ad}$,  $(W(t))_{t \geq 0}$ is a standard Brownian motion on a complete probability space $(\Omega, \mathscr{F}_{T},\textbf{P})$ for $T>0$.

The solution of \eqref{Problem3} can be expressed in the following form of 
    \begin{align} \label{x1}
    x(t)=\mathscr{E}^{A,B}_{h,\alpha}(t)\phi(-h) &+\int_{-h}^{0}\mathscr{E}^{A,B}_{h,\alpha}(t-h-r)\phi^{\prime}(r)dr \nonumber \\
    &+\int_{0}^{t}\mathscr{E}^{A,B}_{h,\alpha,\alpha}(t-h-r)Cu(r)dr+ \int_{0}^{t}\mathscr{E}^{A,B}_{h,\alpha,\alpha}(t-h-r)\Delta(r)dW(r).
    \end{align}
   Define a controllability Grammian matrix $\mathcal{W_{T}}: \mathbb{L}^{2}(\Omega, \mathscr{F}_{T}, \mathbb{R}^{n}) \to \mathbb{L}^{2}(\Omega, \mathscr{F}_{T}, \mathbb{R}^{n})$ as:
   \begin{equation}
   	\mathcal{W}_{T}= \int_{0}^{T}\mathscr{E}^{A,B}_{h, \alpha,\alpha}(T-h-r)CC^{*}\mathscr{E}^{A^{*},B^{*}}_{h, \alpha,\alpha}(T-h-r) \textbf{E}\left\lbrace \cdot|\mathscr{F}_{r}\right\rbrace dr.
   \end{equation}
   
   \begin{thm} \label{singular}
   	The system \eqref{Problem3} is completely controllable on $[0,T]$ if and only if $\mathcal{W}_{T}$ is positive.
   \end{thm}

\begin{proof} 
	\textbf{Sufficiency:} Suppose that $\mathcal{W}_{T}$ is positive, then its inverse is well-defined. For all $\phi(t)$, there exists the control function $u(t)$ defined by:
	
\begin{align} \label{ut1}
u(t)= C^{*}\mathscr{E}^{A^{*},B^{*}}_{h,\alpha,\alpha}(T-h-t) \textbf{E} \Biggl\{ (\mathcal{W_{T}})^{-1}\Big(x_{1}-\mathscr{E}^{A,B}_{h,\alpha}(T)\phi(-h) &-\int_{-h}^{0}\mathscr{E}^{A,B}_{h,\alpha}(T-h-r)\phi^{\prime}(r)dr  \\
&-\int_{0}^{T} \mathscr{E}^{A,B}_{h,\alpha,\alpha}(T-h-r) \Delta(r)dW(r)\Big)| \mathscr{F}_{t} \Biggr\}.\nonumber
\end{align}
Letting $t=T$ in \eqref{x1} and substituting \eqref{ut1} into \eqref{x1}, we get:

\begin{align*}
x(T)= \mathscr{E}^{A,B}_{h,\alpha}(T)\phi(-h)&+\int_{-h}^{0}\mathscr{E}^{A,B}_{h,\alpha}(T-h-r)\phi^{\prime}(r)dr\\
&+\int_{0}^{T} \mathscr{E}^{A,B}_{h,\alpha,\alpha}(T-h-r)C C^{*}\mathscr{E}^{A^{*},B^{*}}_{h,\alpha,\alpha}(T-h-r) \\ &\times \textbf{E}\Biggl\{ (\mathcal{W_{T}})^{-1} \Big( x_{1}-\mathscr{E}^{A,B}_{h,\alpha}(T)\phi(-h)-\int_{-h}^{0}\mathscr{E}^{A,B}_{h,\alpha}(T-h-r)\phi^{\prime}(r)dr  \\
&-\int_{0}^{T} \mathscr{E}^{A,B}_{h,\alpha,\alpha}(T-h-r) \Delta(r)dW(r) \Big)| \mathscr{F}_{r} \Biggr\}dr+\int_{0}^{T} \mathscr{E}^{A,B}_{h,\alpha,\alpha}(T-h-r)\Delta(r)dW(r) \\
&=\mathscr{E}^{A,B}_{h,\alpha}(T)\phi(-h)
+\int_{-h}^{0}\mathscr{E}^{A,B}_{h,\alpha}(T-h-r)\phi^{\prime}(r)dr \\
&+\mathcal{W_{T}} (\mathcal{W_{T}})^{-1}\Biggl\{ x_{1}-\mathscr{E}^{A,B}_{h,\alpha}(T)\phi(-h)  -\int_{-h}^{0}\mathscr{E}^{A,B}_{h,\alpha}(T-h-r)\phi^{\prime}(r)dr  \\
&-\int_{0}^{T} \mathscr{E}^{A,B}_{h,\alpha,\alpha}(T-h-r) \Delta(r)dW(r) \Biggr\}+\int_{0}^{T} \mathscr{E}^{A,B}_{h,\alpha,\alpha}(T-h-r)\Delta(r)dW(r)=x_{1}.
\end{align*}
Hence, $x(T)=x_{1}$. Thus, \eqref{Problem3} is completely controllable. \\

\textbf{Necessity:} Assume that \eqref{Problem3} is completely controllable on $[0,T]$. We have to prove  that $\mathcal{W_{T}}$ is positive. Assume the contrary, there exists a vector $y \neq 0$ such that
	
   	\begin{equation*}
   	\textbf{E}\left\lbrace y^{*}\mathcal{W_{T}}y\right\rbrace=0, \quad y \in \mathbb{L}^{2}(\Omega, \mathscr{F}_{T}, \mathbb{R}^{n}), 
   	\end{equation*}
   	i.e., 
   	\begin{equation*}
   	\textbf{E}\left\lbrace \int_{0}^{T} y^{*}\mathscr{E}^{A,B}_{h,\alpha,\alpha}(T-h-r)C C^{*}\mathscr{E}^{A^{*},B^{*}}_{h,\alpha,\alpha}(T-h-r)\textbf{E}\left\lbrace y|\mathscr{F}_{r}\right\rbrace dr \right\rbrace  =0.
   	\end{equation*}
   	 Then, it follows that
   	 \begin{equation} \label{y7}
   	 \textbf{E}\left\lbrace y^{*} \mathscr{E}^{A,B}_{h,\alpha,\alpha}(T-h-t)C|\mathscr{F}_{t}\right\rbrace=0, \forall t \in [0,T].
   	 \end{equation}
   	 Since the system \eqref{Problem3} is completely controllable, there exist control functions $u_{1}(t)$ and $u_{2}(t)$ such that
   	 \begin{align} \label{ux1}
   	 \mathscr{E}^{A,B}_{h,\alpha}(T)\phi_{1}(-h)
   	 &+\int_{-h}^{0}\mathscr{E}^{A,B}_{h,\alpha}(T-h-r)\phi^{\prime}_{1}(r)dr \nonumber\\
   	  &+\int_{0}^{T}\mathscr{E}^{A,B}_{h,\alpha,\alpha}(T-h-r)Cu_{1}(r)dr+\int_{0}^{T} \mathscr{E}^{A,B}_{h,\alpha,\alpha}(T-h-r)\Delta(r)dW(r)=x_{1}, 
   	 \end{align}
   	 and
   	 \begin{align}\label{ux2}
   	 \mathscr{E}^{A,B}_{h,\alpha}(T)\phi_{2}(-h)
   	 &+\int_{-h}^{0}\mathscr{E}^{A,B}_{h,\alpha}(T-h-r)\phi^{\prime}_{2}(r)dr \nonumber\\
   	 &+\int_{0}^{T}\mathscr{E}^{A,B}_{h,\alpha,\alpha}(T-h-r)Cu_{2}(r)dr+\int_{0}^{T} \mathscr{E}^{A,B}_{h,\alpha,\alpha}(T-h-r)\Delta(r)dW(r)=x_{1}.
   	 \end{align}
   	 From \eqref{ux1} and \eqref{ux2}, we obtain the following expression:
   	 \begin{equation} \label{10}
   	 \int_{0}^{T}\mathscr{E}^{A,B}_{h,\alpha,\alpha}(T-h-r)C\left[u_{2}(r)-u_{1}(r)\right] dr=y,
   	 \end{equation}
   	 where 
   	 \begin{align*}
   	 y \coloneqq \mathscr{E}^{A,B}_{h,\alpha}(T)\left[ \phi_{1}(-h)-\phi_{2}(-h)\right] 
   	 +\int_{-h}^{0}\mathscr{E}^{A,B}_{h,\alpha}(T-h-r)\left[ \phi^{\prime}_{1}(r)- \phi^{\prime}_{2}(r)\right] dr \neq 0.
   	 \end{align*}
   	 Multiplying by $\textbf{E}\left\lbrace y^{*}\right\rbrace$ on the both side of \eqref{10}, we have
   	 \begin{equation*}
   	 \textbf{E}\left\lbrace \int_{0}^{T} y^{*}\mathscr{E}^{A,B}_{h,\alpha,\alpha}(T-h-r)C\left[u_{2}(r)-u_{1}(r)\right] dr\right\rbrace =0.
   	 \end{equation*}
   	 By \eqref{y7}, we acquire $\textbf{E}y^{*}y=0$, which is contradiction to $y\neq 0$.
   	 The proof is complete.
   	\end{proof}
   	 Furthermore, we prove the complete controllability results in the following theorem by using the rank correlation of Cayley-Hamilton theorem. In this case, we need to formulate on algebraic condition equivalent to controllability. For matrices $A,B \in \mathbb{R}^{n\times n}$ and $C \in \mathbb{R}^{n \times m}$ denote by $H_{n}$, where the matrix
   	 \begin{equation*}
   	 H_{n} \coloneqq \Biggl\{C| AC| A^{2}C|\cdots |A^{n-1}C |BC| ABC | A^{2}BC| \cdots |A^{n-1}BC|\cdots |B^{n-1}C|AB^{n-1}C| \cdots | A^{n-1}B^{n-1}C \Biggr\},
   	 \end{equation*}
   	 which consists of consecutively written columns of matrices $C, AC, A^{2}C, \cdots, A^{n-1}B^{n-1}C$. 
   	 \begin{thm} \label{thm1}
   	 	The following statements are equivalent: \\
   	 	(i) System \eqref{Problem3} is completely controllable ;\\
   	 	(ii) System \eqref{Problem3} is completely controllable at a given time $T \geq (n-1)h$ ;\\
   	 	(iii) Grammian matrix $\mathcal{W_{T}}$ is positive for an arbitrary $T>0$; \\
   	 	(iv) $rank H_{n}=n$. \\
   	 	 \end{thm}
   	Condition (iv) is called the $Kalman$ $rank$ $condition$. To apply Cayley-Hamilton theorem, first, we determine the characteristic polynomial $p(\cdot)$ of a matrix $A \in \mathbb{R}^{n \times n}$ is defined by:
   	\begin{equation}\label{poly}
   	p(\lambda)= \det [\lambda I-A], \quad \lambda \in \mathbb{C},
   	\end{equation} 
  where $I \in \mathbb{R}^{n \times n}$ is the identity matrix.
   	Let 
   	\begin{equation}\label{plambda}
   	p(\lambda)= \lambda^{n}+a_{1}\lambda^{n-1}+ \cdots+ a_{n}, \quad  \lambda \in \mathbb{C}.
   	\end{equation}
   	The Cayley-Hamilton theorem has the following formulation in the next theorem \cite{balakrishnan}.
   	\begin{thm} [Cayley-Hamilton]\label{Cayley}
   		For arbitrary $A \in \mathbb{R}^{n \times n}$, with the characteristic polynomial \eqref{plambda},
   		\begin{equation}
   		A^{n}+a_{1}A^{n-1}+\cdots+ a_{n}I=0.
   		\end{equation}
   		Symbolically, $p(A)=0$.
   	\end{thm}
   \textbf{Proof of Theorem \ref{thm1}}. Equivalencies of (i)-(iii) follow from the proof of Theorem \ref{singular} and the following identity for $L_{T}: U_{ad} to \mathbb{L}^{2}(\Omega, \mathscr{F}_{T}, \mathbb{R}^{n})$, 
   	\begin{equation} \label{identity}
   	x(T)=L_{T}u + \mathfrak{G}, 
   	\end{equation}
   	where
   	\begin{align*} 
   	  x(T)=\mathscr{E}^{A,B}_{h,\alpha}(T)\phi(-h) &+\int_{-h}^{0}\mathscr{E}^{A,B}_{h,\alpha}(T-h-r)\phi^{\prime}(r)dr \nonumber \\
   	  &+\int_{0}^{T}\mathscr{E}^{A,B}_{h,\alpha,\alpha}(T-h-r)Cu(r)dr+ \int_{0}^{T}\mathscr{E}^{A,B}_{h,\alpha,\alpha}(T-h-r)\Delta(r)dW(r),
   	  \end{align*}
   	 	 and 
   	 	 \begin{align*} 
   	 	 \mathfrak{G}=\mathscr{E}^{A,B}_{h,\alpha}(T)\phi(-h) &+\int_{-h}^{0}\mathscr{E}^{A,B}_{h,\alpha}(T-h-r)\phi^{\prime}(r)dr \nonumber \\
   	 	 &+\int_{0}^{T}\mathscr{E}^{A,B}_{h,\alpha,\alpha}(T-h-r)\Delta(r)dW(r).
   	 	 \end{align*}
   	 	 From \eqref{identity},  by change of variable $T-h-r= \mu,$ we have  
   	 	 \begin{align} \label{x-P}
   	 	 L_{T}u=x(T)-\mathfrak{G}&= \int_{0}^{T}\mathscr{E}^{A,B}_{h,\alpha,\alpha}(T-h-r)Cu(r)dr
   	 	 =\int_{-h}^{T-h}\mathscr{E}^{A,B}_{h,\alpha,\alpha}(\mu)Cu(T-h-\mu)d\mu \nonumber\\
   	 	 &=\int_{-h}^{0}(\mu+h)^{\alpha-1}\Big[I\frac{1}{\Gamma(\alpha)}+A\frac{(\mu+h)^{\alpha}}{\Gamma(2\alpha)}+\cdots+A^{l-1}\frac{(\mu+h)^{(l-1)\alpha}}{\Gamma(l\alpha)}\Big]Cu(T-h-\mu)d\mu \nonumber\\
   	 	 &+\int_{0}^{h}\Big[(\mu+h)^{\alpha-1}\Big(I\frac{1}{\Gamma(\alpha)}+A\frac{(\mu+h)^{\alpha}}{\Gamma(2\alpha)}+\cdots+A^{l-1}\frac{(\mu+h)^{(l-1)\alpha}}{\Gamma(l\alpha)}\Big)\nonumber\\
   	 	 &+\mu^{2\alpha-1}\Big(B\frac{1}{\Gamma(2\alpha)} +AB\frac{2\mu^{\alpha}}{\Gamma(3\alpha)}+\cdots+A^{l-1}B\frac{l\mu^{(l-1)\alpha}}{\Gamma((l+1)\alpha)}\Big)             \Big]Cu(T-h-\mu)d\mu \nonumber\\
   	 	 & +\cdots+  \int_{(l-2)h}^{T-h}\Big[ (\mu+h)^{\alpha-1}\Big(I\frac{1}{\Gamma(\alpha)}+A\frac{(\mu+h)^{\alpha}}{\Gamma(2\alpha)}+\cdots+A^{l-1}\frac{(\mu+h)^{(l-1)\alpha}}{\Gamma(l\alpha)}\Big)\nonumber
   	 	 \end{align}
   	 	 \begin{align}
   	 	 &+\mu^{2\alpha-1}\Big(B\frac{1}{\Gamma(2\alpha)} +AB\frac{2\mu^{\alpha}}{\Gamma(3\alpha)}+\cdots+A^{l-1}B\frac{l\mu^{(l-1)\alpha}}{\Gamma((l+1)\alpha)}\Big) \nonumber \\
   	 	 &+\cdots + (\mu-(l-2)h)^{l\alpha-1}\Big(B^{l-1}\frac{1}{\Gamma(l\alpha)}+AB^{l-1} \begin{pmatrix}l \\ 1\end{pmatrix}\frac{(\mu-(l-2)h)^{\alpha}}{\Gamma((l+1)\alpha)} 
   	 	 \nonumber\\
   	 	 &+\cdots+A^{l-1}B^{l-1}\begin{pmatrix} 2l-2 \\ l-1 \end{pmatrix}\frac{(\mu-(l-2)h)^{(l-1)\alpha}}{\Gamma((2l-1)\alpha)}     \Big)\Big]Cu(T-h-\mu)d\mu\nonumber\\
   	 	 &=\int_{-h}^{0}\Big[I\frac{(\mu+h)^{\alpha-1}}{\Gamma(\alpha)}+A\frac{(\mu+h)^{2\alpha-1}}{\Gamma(2\alpha)}+\cdots+A^{l-1}\frac{(\mu+h)^{l\alpha-1}}{\Gamma(l\alpha)}\Big]Cu(T-h-\mu)d\mu \\
   	 	 &+\int_{0}^{h}\Big[I\frac{(\mu+h)^{\alpha-1}}{\Gamma(\alpha)}+A\frac{(\mu+h)^{2\alpha-1}}{\Gamma(2\alpha)}+\cdots+A^{l-1}\frac{(\mu+h)^{l\alpha-1}}{\Gamma(l\alpha)} \nonumber\\
   	 	 &+B\frac{\mu^{2\alpha-1}}{\Gamma(2\alpha)}+AB\frac{2\mu^{3\alpha-1}}{\Gamma(3\alpha)}+\cdots+A^{l-1}B\frac{l\mu^{(l+1)\alpha-1}}{\Gamma((l+1)\alpha)} \Big]Cu(T-h-\mu)d\mu \nonumber\\
   	 	 &+\cdots+ \int_{(l-2)h}^{T-h}\Big[I\frac{(\mu+h)^{\alpha-1}}{\Gamma(\alpha)}+A\frac{(\mu+h)^{2\alpha-1}}{\Gamma(2\alpha)}+\cdots+A^{l-1}\frac{(\mu+h)^{l\alpha-1}}{\Gamma(l\alpha)} \nonumber \\
   	 	 &+B\frac{\mu^{2\alpha-1}}{\Gamma(2\alpha)}+AB\frac{2\mu^{3\alpha-1}}{\Gamma(3\alpha)}+\cdots+A^{l-1}B\frac{l\mu^{(l+1)\alpha-1}}{\Gamma((l+1)\alpha)} \nonumber\\
   	 	 &+\cdots+ B^{l-1}\frac{(\mu-(l-2)h)^{l\alpha-1}}{\Gamma(l\alpha)} +AB^{l-1} \begin{pmatrix}l \nonumber\\ 1\end{pmatrix}\frac{(\mu-(l-2)h)^{(l+1)\alpha-1}}{\Gamma((l+1)\alpha)} 
   	 	 \nonumber\\
   	 	 &+\cdots+A^{l-1}B^{l-1}\begin{pmatrix} 2l-2 \\ l-1 \end{pmatrix}\frac{(\mu-(l-2)h)^{(2l-1)\alpha}}{\Gamma((2l-1)\alpha-1)}\Big]Cu(T-h-\mu)d\mu. \nonumber
   	 	 \end{align}

   	 	 If we denote 
   	 	 \begin{align} \label{psi}
   	 	 &\Psi_{11}(T)=\int_{-h}^{T-h}\frac{(\mu+h)^{\alpha-1}}{\Gamma(\alpha)}u(T-h-\mu)d\mu ,\nonumber\\
   	 	 &\Psi_{21}(T)=\int_{-h}^{T-h}\frac{(\mu+h)^{2\alpha-1}}{\Gamma(2\alpha)}u(T-h-\mu)d\mu ,\nonumber\\
   	 	 & \cdots \nonumber\\
   	 	 &\Psi_{m1}(T)=\int_{-h}^{T-h}\frac{(\mu+h)^{l\alpha-1}}{\Gamma(l\alpha)}u(T-h-\mu)d\mu; \nonumber \\ 
   	 	 &\Psi_{12}(T)=\int_{0}^{T-h}\frac{\mu^{2\alpha-1}}{\Gamma(2\alpha)}u(T-h-\mu)d\mu ,\nonumber\\
   	 	 &\Psi_{22}(T)=\int_{0}^{T-h}\frac{2\mu^{3\alpha-1}}{\Gamma(3\alpha)}u(T-h-\mu)d\mu , \nonumber\\
   	 	 & \cdots \nonumber\\
   	 	 &\Psi_{m2}(T)=\int_{0}^{T-h}\frac{l\mu^{(l+1)\alpha-1}}{\Gamma((l+1)\alpha)}u(T-h-\mu)d\mu ; \\ 
   	 	 & \cdots \nonumber \\
   	 	 &\Psi_{1l}(T)=\int_{(l-2)h}^{T-h}\frac{(\mu-(l-2)h)^{l\alpha-1}}{\Gamma(k\alpha)}u(T-h-\mu)d\mu,\nonumber \\
   	 	 &\Psi_{2l}(T)=\int_{(l-2)h}^{T-h}\begin{pmatrix}
   	 	 l \\ 1
   	 	 \end{pmatrix}\frac{(\mu-(l-2)h)^{(l+1)\alpha-1}}{\Gamma((l+1)\alpha)}u(T-h-\mu)d\mu,\nonumber \\
   	 	 &\cdots \nonumber \\
   	 	 &\Psi_{ml}(T)=\int_{(l-2)h}^{T-h}\begin{pmatrix}
   	 	 2l-2 \\ l-1
   	 	 \end{pmatrix}\frac{(\mu-(l-2)h)^{(2l-1)\alpha-1}}{\Gamma((2l-1)\alpha)}u(T-h-\mu)d\mu.\nonumber 
   	 	 \end{align}
   	 	 
   	 	 Then, using \eqref{psi}, \eqref{x-P} can be written as below:
   	 	 \begin{align}\label{matrix}
   	 	 & C\Psi_{11}(T)+BC\Psi_{12}(T)+ \cdots + B^{l-1}C\Psi_{1l}(T)+AC\Psi_{21}(T) \\ \nonumber
   	 	 &+ ABC\Psi_{22}(T) +\cdots + AB^{l-1}C\Psi_{2l}(T)  + A^{2}C\Psi_{31}(T) \\ \nonumber
   	 	 &+ A^{2}BC\Psi_{32}(T) + \cdots  + A^{l-1}C\Psi_{m1}(T)
   	 	  +A^{l-1}BC\Psi_{m2}(T) +\cdots+ A^{l-1}B^{l-1}C\Psi_{ml}(T)=L_{T}u.
   	 	 \end{align}
   	 	 Because the linear system \eqref{Problem3} is completely controllable, then \eqref{matrix} has a solution for an arbitrary vector $\nu$. If $l < n$, then the system is overdetermined and does not always have a solution. Therefore, for the system \eqref{Problem3} to be completely controllable, it is necessary that $T > (l-1)h\geq (n-1)h$. \\ 
   	 	 To show equivalencies for condition (iv), it is convenient to introduce a linear mapping $\mathcal{L}_{n}$ from  $\mathbb{L}^{2}(\Omega, \mathscr{F}_{T}, \mathbb{R}^{m})$ into $\mathbb{L}^{2}(\Omega, \mathscr{F}_{T}, \mathbb{R}^{n})$ according to the relation \eqref{matrix}
   	 	 \begin{align*}
   	 	 \mathcal{L}_{n}(u_{0},u_{1}, \cdots, u_{n-1})&= \sum_{j=0}^{n-1}A^{j}(I+B+\cdots+B^{n-1})Cu_{j}\\
   	 	 &=\sum_{j=0}^{n-1}A^{j}Cu_{j}+\sum_{j=0}^{n-1}A^{j}BCu_{j}+\cdots+\sum_{j=0}^{n-1}A^{j}B^{n-1}Cu_{j},  u_{j} \in \mathbb{L}^{2}(\Omega, \mathscr{F}_{T}, \mathbb{R}^{n}), j=0,...,n-1.
   	 	 \end{align*}
   	 	 We need to prove first the following lemma.
   	 	 \begin{lem}
   	 	 	The transformation $L_{T}$, $T>0$, has the same image as $\mathcal{L}_{n}$. In particular, $L_{T}$ is onto if and only if $\mathcal{L}_{n}$ is onto. 
   	 	 \end{lem}
   	 	 	\begin{proof}
   	 	 		For arbitrary $v \in \mathbb{L}^{2}(\Omega, \mathscr{F}_{T}, \mathbb{R}^{n})$, $u \in U_{ad}$, $u_{i} \in \mathbb{L}^{2}(\Omega, \mathscr{F}_{T}, \mathbb{R}^{m}), i=0,1,\cdots, n-1$,
   	 	 		\begin{align*}
   	 	 		&\textbf{E} \langle L_{T}u,v \rangle=\textbf{E} \int_{0}^{T} \langle u(r), C^{*}\mathscr{E}^{A^{*},B^{*}}_{h,\alpha,\alpha}(T-h-r)v(r)\rangle dr, \\
   	 	 		&\textbf{E} \langle \mathcal{L}_{n}(u_{0}, u_{1},\cdots, u_{n-1}), v \rangle= \textbf{E} \langle u_{0}, C^{*}v\rangle +\textbf{E} \langle u_{0}, C^{*}A^{*}v \rangle \\
   	 	 		&+ \cdots + \textbf{E}\langle u_{0},C^{*}(A^{*})^{n-1}v\rangle + \textbf{E} \langle u_{1}, C^{*}B^{*}v \rangle +\textbf{E} \langle u_{1},C^{*}B^{*}A^{*}v\rangle\\
   	 	 		&+\cdots +\textbf{E} \langle u_{1},C^{*}B^{*}(A^{*})^{n-1}v\rangle \\
   	 	 		&+\cdots + \textbf{E}\langle u_{n-1}, C^{*}(B^{*})^{n-1}v\rangle+\cdots + \textbf{E}\langle u_{n-1}, C^{*}(B^{*})^{n-1}(A^{*})^{n-1}v\rangle.
   	 	 		\end{align*}
   	 	 		Assume that $\textbf{E} \langle \mathcal{L}_{n}(u_{0}, u_{1},\cdots, u_{n-1}),v \rangle =0$ for arbitrary  $u_{0}, u_{1},\cdots, u_{n-1} \in \mathbb{L}^{2}(\Omega, \mathscr{F}_{T},\mathbb{R}^{n}) $. Then,
   	 	 		\begin{align*}
   	 	 		&C^{*}v=0, C^{*}A^{*}v=0, C^{*}(A^{*})^{n-1}v=0, C^{*}B^{*}v=0, C^{*}B^{*}A^{*}v=0, \cdots, \\ &C^{*}B^{*}(A^{*})^{n-1}v=0, \cdots,C^{*}(B^{*})^{n-1}v=0, \cdots, C^{*}(B^{*})^{n-1}(A^{*})^{n-1}v = 0.
   	 	 		\end{align*}

   	 	 		From Theorem \ref{Cayley}, applied to matrix $A^{*}$, it follows that for some constants $c_{0},c_{1},\cdots, c_{n-1} :$
   	 	 		\begin{equation*}
   	 	 		(A^{*})^{n}=\sum_{m=0}^{n-1}c_{m}(A^{*})^{m}.
   	 	 		\end{equation*}
   	 	 		Thus, by induction, for arbitrary $p=0,1,2,...$ there exist constants $c_{p,0}, c_{p,1},\cdots, c_{p,n-1}$ such that
   	 	 		\begin{equation*}
   	 	 		(A^{*})^{n+p}=\sum_{m=0}^{n-1}c_{p,m}(A^{*})^{m}.
   	 	 		\end{equation*}
   	 	 		Therefore, $C^{*}(A^{*})^{m}v=0, C^{*}B^{*}(A^{*})^{m}v=0, \cdots, C^{*}(B^{*})^{n-1}(A^{*})^{m}v=0$ for $m=0,1,2,..$. Taking into account
   	 	 		\begin{align*}
   	 	 		C^{*}\mathscr{E}^{A^{*},B^{*}}_{h,\alpha,\alpha}(t)v=&\sum_{i=0}^{n-1}(t-(i-1)h)^{(i+1)\alpha-1}C^{*}(B^{*})^{i}E^{i+1}_{\alpha, (i+1)\alpha}(A{(t-(i-1)h)^{\alpha}})v\\
   	 	 		&=(t+h)^{\alpha-1}C^{*}E^{1}_{\alpha,\alpha}(A^{*}(t+h)^{\alpha})v+t^{2\alpha-1}C^{*}B^{*}E^{2}_{\alpha,2\alpha}(A^{*}t^{\alpha})v\\
   	 	 		&+ \cdots+ (t-(n-2)h)^{n\alpha-1}C^{*}(B^{*})^{n-1}E^{n}_{\alpha,n\alpha}(A^{*}(t-(n-2)h)^{\alpha})v=0.
   	 	 		\end{align*}
   	 	 		We deduce that for an arbitrary $T>0$ and $t \in [0, T]$,
   	 	 		\begin{equation}
   	 	 		C^{*}\mathscr{E}^{A^{*},B^{*}}_{h,\alpha,\alpha}(t)v=0,
   	 	 		\end{equation}
   	 	 		so, $\textbf{E} \langle L_{T}u ,v \rangle=0$ for arbitrary $u\in U_{ad}$.
   	 	 		Suppose, conversely, that for arbitrary $u\in U_{ad}$, $\textbf{E} \langle L_{T}u ,v \rangle=0$. Then $C^{*}\mathscr{E}^{A^{*},B^{*}}_{h,\alpha,\alpha}(t)v=0$ for $t\in [0, T]$. Caputo type differentiating  $(k+1)\alpha-1$,$(k+2)\alpha-1$,$\cdots$,$(k+n)\alpha-1$-times for $k=0,1,2,...$ the following identity
   	 	 		\begin{equation*}
   	 	 		\sum_{i=0}^{n-1}(t-(i-1)h)^{(i+1)\alpha-1}C^{*}(B^{*})^{i}E^{i+1}_{\alpha, (i+1)\alpha}(A^{*}(t-(i-1)h)^{\alpha})v=0, \quad t \in [0, T],
   	 	 		\end{equation*}
   	 	 		where 
   	 	 		\begin{equation*}
   	 	 		E^{i+1}_{\alpha, (i+1)\alpha}(A^{*}(t-(i-1)h)^{\alpha})=\sum_{k=0}^{\infty}\begin{pmatrix}
   	 	 		k+i \\ k
   	 	 		\end{pmatrix} \frac{(A^{*})^{k}(t-(i-1)h)^{k\alpha}}{\Gamma((i+k+1)\alpha)}
   	 	 		\end{equation*}
   	 	 		and inserting each time $t=0$, we attain that 
   	 	 		\begin{equation*}
   	 	 		C^{*}(A^{*})^{m}v=0, C^{*}B^{*}(A^{*})^{m}v=0, \cdots, C^{*}(B^{*})^{n-1}(A^{*})^{m}v=0, \quad  m=0,1,2,...,n-1.
   	 	 		\end{equation*}
   	 	 		Therefore,
   	 	 		\begin{equation*}
   	 	 		\langle \mathcal{L}_{n}(u_{0},u_{1}, \cdots, u_{n-1}), v \rangle =0, \quad \forall u_{0}, u_{1}, ... , u_{n-1} \in \mathbb{L}^{2}(\Omega, \mathscr{F}_{T},\mathbb{R}^{n}) .
   	 	 		\end{equation*}
   	 	 	\end{proof}
   	 	 Assume that \eqref{Problem3} is complete controllable.Then the transformation $L_{T}$ is onto $\mathbb{R}^{n}$ for arbitrary $T>0$ and by above lemma, the matrix $H_{n}$ has rank $n$. In contrast, if the rank of $H_{n}$ is n, then the mapping $\mathcal{L}_{n}$ is onto $\mathbb{R}^{n}$ and also, the transformation $L_{T}$ is onto $\mathbb{R}^{n}$, thus the controllability of \eqref{Problem3} follows. Proof is complete. $\blacksquare$

   \subsection{Nonlinear case}
   
   Consider nonlinear system corresponding to \eqref{Problem3} on $[0,T]$:
   \begin{equation} \label{Problem4}
   \begin{cases}
   (\prescript{C}{}D^{\alpha}_{0^{+}}x)(t)=Ax(t)+Bx(t-h)+Cu(t)+\Delta(t,x(t))\frac{dW(t)}{dt}, \quad h >0,\\
   x(t)=\phi(t), \quad t \in [-h, 0],
   \end{cases}
   \end{equation}
   where  $\Delta \in [0,T] \times \mathbb{R}^{n} \to \mathbb{R}^{n}$ is measurable and bounded function.\\
   
   The solution of \eqref{Problem4} can be expressed in the following form of 
   \begin{align} \label{x}
   x(t)=\mathscr{E}^{A,B}_{h,\alpha}(t)\phi(-h) &+\int_{-h}^{0}\mathscr{E}^{A,B}_{h,\alpha}(t-h-r)\phi^{\prime}(r)dr \nonumber \\
   &+\int_{0}^{t}\mathscr{E}^{A,B}_{h,\alpha,\alpha}(t-h-r)Cu(r)dr+ \int_{0}^{t}\mathscr{E}^{A,B}_{h,\alpha,\alpha}(t-h-r)\Delta(r,x(r))dW(r).
   \end{align}
   
Define the operator $L_{T}: \mathbb{L}^{2}(\Omega, \mathscr{F}_{T},\mathbb{R}^{n}) \to \mathbb{L}^{2}(\Omega, \mathscr{F}_{T},\mathbb{R}^{n})$ as
\begin{equation} \label{Lt}
L_{T}u=\int_{0}^{T}\mathscr{E}^{A,B}_{h,\alpha,\alpha}(T-h-r)Cu(r)dr.
\end{equation} 
Clearly, the adjoint operator $L^{*}_{T}$ of $L_{T}$ satisfying $L^{*}_{T} \in \mathbf{L}(\mathbb{L}^{2},\mathbb{L}^{2})$ is obtained as below:
\begin{equation*}
(L^{*}_{T}x)(t)= C^{*}\mathscr{E}^{A^{*},B^{*}}_{h,\alpha,\alpha}(T-h-t)\textbf{E}\left\lbrace x|\mathscr{F}_{t}\right\rbrace. 
\end{equation*} 

\begin{defn}
	The controllability Grammian operator $\mathcal{W}_{T} : \mathbb{L}^{2}(\Omega, \mathscr{F}_{T},\mathbb{R}^{n}) \to \mathbb{L}^{2}(\Omega, \mathscr{F}_{T},\mathbb{R}^{n})$ is defined by 
	\begin{equation} \label{Wz}
	\mathcal{W}_{T}z= \int_{0}^{T}\mathscr{E}^{A,B}_{h, \alpha,\alpha}(T-h-r)CC^{*}\mathscr{E}^{A^{*},B^{*}}_{h, \alpha,\alpha}(T-h-r)\textbf{E}\left\lbrace z|\mathscr{F}_{t}\right\rbrace dr. 
	\end{equation}
\end{defn}
The corresponding deterministic operator $\Gamma_{T-r}: \mathbb{R}^{n} \to \mathbb{R}^{n}$ is given by
\begin{equation*}
\Gamma_{T-r}x= \int_{r}^{T}\mathscr{E}^{A,B}_{h, \alpha,\alpha}(T-h-r)CC^{*}\mathscr{E}^{A^{*},B^{*}}_{h, \alpha,\alpha}(T-h-r)x dr.
\end{equation*}
\begin{thm}
	The fractional stochastic system \eqref{Problem3} is completely controllable on $[0,T]$ if and only if for some $\gamma >0$
	\begin{equation} \label{Omega}
	\textbf{E}\langle\mathcal{W}_{T}x,x \rangle \geq \gamma \textbf{E}\|x\|^{2}, \forall x \in \mathbb{L}^{2}(\Omega, \mathscr{F}_{T},\mathbb{R}^{n}). \\ 
	\end{equation}
\end{thm}
\begin{lem}
	For every $z \in \mathbb{L}^{2}(\Omega, \mathscr{F}_{T},\mathbb{R}^{n})$, there exists a predictable $\mathbb{L}^{2}$ process $\varphi(\cdot)$ such that \\
	(i)  $z= \textbf{E}z+\int_{0}^{T}\varphi(r)dW(r)$, \\
    (ii)  $\mathcal{W_{T}}z= \Gamma_{T}\textbf{E}z+ \int_{0}^{T}\Gamma_{T-r}\varphi(r)dW(r)$.

\end{lem}

	\begin{proof}
		The proof of (i) can be found in \cite{mahmudov1}.\\
		(ii) Let $z\in \mathbb{L}^{2}( \Omega, \mathscr{F}_{T}, \mathbb{R}^{n} )$, then we have 
		\begin{equation*}
		\textbf{E}\left\lbrace z |\mathscr{F}_{t}\right\rbrace =\textbf{E}z+\int_{0}^{t}\varphi(r)dW(r).
		\end{equation*}
	The definition of deterministic  operator and stochastic Fubini's theorem lead to the desired representation:
	\allowdisplaybreaks
	\begin{align*}
	\mathcal{W}_{T}z&= \int_{0}^{T}\mathscr{E}^{A,B}_{h, \alpha,\alpha}(T-h-t)CC^{*}\mathscr{E}^{A^{*},B^{*}}_{h, \alpha,\alpha}(T-h-t)\textbf{E}\left\lbrace z|\mathscr{F}_{t}\right\rbrace dt\\
	&= \int_{0}^{T}\mathscr{E}^{A,B}_{h, \alpha,\alpha}(T-h-t)CC^{*}\mathscr{E}^{A^{*},B^{*}}_{h, \alpha,\alpha}(T-h-t) \Big[\textbf{E}z+\int_{0}^{t}\varphi(r)dW(r)\Big]dt\\
	&=\int_{0}^{T}\mathscr{E}^{A,B}_{h, \alpha,\alpha}(T-h-t)CC^{*}\mathscr{E}^{A^{*},B^{*}}_{h, \alpha,\alpha}(T-h-t)\textbf{E}zdt \\
	&+ \int_{0}^{T}\int_{r}^{T}\mathscr{E}^{A,B}_{h, \alpha,\alpha}(T-h-t)CC^{*}\mathscr{E}^{A^{*},B^{*}}_{h, \alpha,\alpha}(T-h-t)\varphi(r)dtdW(r)\\
	&=\Gamma_{T}\textbf{E}z+\int_{0}^{T}\Gamma_{T-r}\varphi(r)dW(r).  
	\end{align*}
	This completes the proof of lemma.
	\end{proof}

Using representation of (i) and (ii) of above lemma, we write $\textbf{E}\langle \mathcal{W_{T}}z, z \rangle $ in terms of $\langle \Gamma_{T}\textbf{E}z,\textbf{E}z \rangle$ and using inequality \eqref{Omega} and scalar product of stochastic integral to show coercivity of $\mathcal{W_{T}}$ , for all $z \in \mathbb{L}^{2}( \Omega, \mathscr{F}_{T}, \mathbb{R}^{n} )$ \cite{mahmudov1},

 \begin{align*}
 \textbf{E}\langle \mathcal{W_{T}}z, z \rangle &= \textbf{E} \Big \langle\Gamma_{T}\textbf{E}z + \int_{0}^{T}\Gamma_{T-r}\varphi(r)dW(r), \textbf{E}z +\int_{0}^{T}\varphi(r)dW(r) \Big \rangle \\
 &=\langle \Gamma_{T}\textbf{E}z,\textbf{E}z \rangle +\textbf{E}\int_{0}^{T}\langle  \Gamma_{T-r}\varphi(r),\varphi(r) \rangle dr\\
 & \geq \gamma \Big( \|\textbf{E}z\|^{2}+\textbf{E}\int_{0}^{T}\|\varphi(r)\|^{2}dr\Big)=\gamma\|z\|^{2}.
 \end{align*}

   We set following hypotheses to derive sufficient conditions for stochastic controlability of nonlinear system \eqref{Problem4}. \\
   (H1) Controllability operator $\mathcal{W}_{T}$ has its inverse $(\mathcal{W}_{T})^{-1}$. Then we take 
   \begin{equation*}
   k_{1}=\textbf{E}\|(\mathcal{W}_{T})^{-1}\|^{2}.\\
   \end{equation*}
   (H2) There exists $L_{\Delta}>0$ such that for all $x,y \in \mathbb{R}^{n}, t \in [0,T]$, 
   \begin{equation*} 
   \|\Delta(t.x)-\Delta(t,y)\|^{2}\leq L_{\Delta}\|x-y\|^{2},
   \end{equation*}
   (H3) Let $\lambda \coloneqq 16N^{2}\|C\|^{2}\|L_{T}^{*}\|^{2}k_{1}$ be such that $0 \leq\lambda <1$, and let $C_{1}$ and $C_{2}$ be such that, \linebreak $C_{1} \coloneqq M^{2}(4+\lambda)(1+K^{2})\textbf{E}\|\phi(-h)\|^{2}$ and $C_{2} \coloneqq N^{2}L^{2}_{\Delta}(4+\lambda K^{2})T$,\\
   (H4) Let $\rho = N^{2}L_{\Delta}^{2}T$ be such that $0 \leq \rho <1$.
   
   \begin{thm}
   	Suppose that hypotheses (H1)-(H4) hold. If the linear fractional stochastic delay system \eqref{Problem3} is completely controllable, then the nonlinear fractional stochastic delay system \eqref{Problem4} is completely controllable.
   \end{thm}
  \begin{proof}
  	Let $x_{1}$ be an arbitrary random variable in $\mathbb{L}^{2}(\Omega, \mathscr{F}_{T}, \mathbb{R}^{n})$. We focus on theoretical results to a stochastic controllability of fractional delay system via fixed point technique. Define an operator \linebreak $\Phi: H^{2} \to H^{2}$ by
  	\begin{align} \label{Tx}
  	(\Phi x)(t)=\mathscr{E}^{A,B}_{h,\alpha}(t)\phi(-h) &+\int_{-h}^{0}\mathscr{E}^{A,B}_{h,\alpha}(t-h-r)\phi^{\prime}(r)dr \nonumber \\
  	&+\int_{0}^{t}\mathscr{E}^{A,B}_{h,\alpha,\alpha}(t-h-r)Cu(r)dr \\
  	&+\int_{0}^{t}\mathscr{E}^{A,B}_{h,\alpha,\alpha}(t-h-r)\Delta(r,x(r))dW(r)\nonumber.
  	\end{align}
  Since the linear system \eqref{Problem3} is controllable, we have that $\mathcal{W_{T}}$ is invertible. We define control function $u$ as 
   
\begin{align} \label{ut}
u(t)= C^{*}\mathscr{E}^{A^{*},B^{*}}_{h,\alpha,\alpha}(T-h-t) &\textbf{E} \Biggl\{(\mathcal{W_{T}})^{-1} \Big(x_{1}-\mathscr{E}^{A,B}_{h,\alpha}(T)\phi(-h) \nonumber\\ 
&-\int_{-h}^{0}\mathscr{E}^{A,B}_{h,\alpha}(T-h-r)\phi^{\prime}(r)dr  \\
&-\int_{0}^{T} \mathscr{E}^{A,B}_{h,\alpha,\alpha}(T-h-r) \Delta(r, x(r))dW(r) \Big) | \mathscr{F}_{t} \Biggr\}.\nonumber
\end{align}
Now, we need to show that $\Phi$ has a fixed point. This fixed point is a solution of control problem. Clearly, \linebreak $\Phi(x(T))=x_{1}$, which means that the control $u$ steers the nonlinear system from initial state $x_{0}$ to $x_{1}$ in the time $T$, provided that we can obtain a fixed point of nonlinear operator $\Phi$. To verify the conditions for Banach contraction principle, we divide our proof into two steps. 

\textbf{Step 1.} First, we prove that $\Phi$ maps from $H^{2}$ into itself.
Using hypotheses (H1)-(H3), we have 

\begin{align} \label{main} 
\textbf{E}\| (\Phi x) (t)\|^{2}= 4\textbf{E}\|\mathscr{E}^{A,B}_{h,\alpha}(t)\phi(-h)\|^{2}
&+4\textbf{E}\|\int_{-h}^{0}\mathscr{E}^{A,B}_{h,\alpha}(t-h-r)\phi^{\prime}(r)dr\|^{2} \nonumber\\
&+4\textbf{E}\|\int_{0}^{t}\mathscr{E}^{A,B}_{h,\alpha,\alpha}(t-h-r)Cu(r)dr\|^{2}  \\
&+4\textbf{E}\| \int_{0}^{t}\mathscr{E}^{A,B}_{h,\alpha,\alpha}(t-h-r)\Delta(r,x(r))dW(r)\|^{2} \nonumber \\
& \coloneqq 4(I_{1}+I_{2}+I_{3}+I_{4}). \nonumber 
\end{align}
For convenience, let us introduce following constants:
\begin{align} \label{const1}
M=\sup_{t \in [0,T]}\|\mathscr{E}^{A,B}_{h,\alpha}(t)\|; \quad  N=\sup_{t \in [0,T]}\|\mathscr{E}^{A,B}_{h,\alpha, \alpha}(t)\|.
\end{align} 
 There exists $K>0$,
 \begin{equation} \label{relation}
 \textbf{E}\|\phi(0)-\phi(-h)\|^{2}\leq K^{2}\textbf{E}\|\phi(-h)\|^{2}.
 \end{equation} 
 We impose standard computations by taking into \eqref{const1} and \eqref{relation} account in a following order:
\begin{equation} \label{1}
I_{1} \coloneqq \textbf{E}\|\mathscr{E}^{A,B}_{h,\alpha}(t)\phi(-h)\|^{2} \leq M^{2}\textbf{E}\|\phi(-h)\|^{2},
\end{equation}	
\begin{equation} \label{2}
I_{2} \coloneqq \textbf{E}\|\int_{-h}^{0}\mathscr{E}^{A,B}_{h,\alpha}(t-h-r)\phi^{\prime}(r)dr\|^{2} \leq M^{2}\textbf{E}\|\phi(0)-\phi(-h)\|^{2} \leq M^{2}K^{2}\textbf{E}\|\phi(-h)\|^{2}.
\end{equation} 
Applying hypotheses (H1) and (H2), we define $u(t)$ as follows :
\begin{align*} 
\|u(t)\|^{2} &\leq \|C^{*}\mathscr{E}^{A,B}_{h,\alpha,\alpha}(T-h-t) \textbf{E} \Biggl\{(\mathcal{W_{T}})^{-1} \Big(x_{1}-\mathscr{E}^{A,B}_{h,\alpha}(T)\phi(-h)  
-\int_{-h}^{0}\mathscr{E}^{A,B}_{h,\alpha}(T-h-r)\phi^{\prime}(r)dr \\
&-\int_{0}^{T} \mathscr{E}^{A,B}_{h,\alpha,\alpha}(T-h-r) \Delta(r, x(r))dW(r) \Big) | \mathscr{F}_{t} \Biggr\}\|^{2}\\
& \leq 4 \|C\|^{2} N^{2}\textbf{E}\|(\mathcal{W_{T}})^{-1}\|^{2}\Big[
\textbf{E}\|x_{1}\|^{2}+M^{2}\textbf{E}\|\phi(-h)\|^{2}\\
&+M^{2}\int_{-h}^{0}\textbf{E}\|\phi^{\prime}(r)\|^{2}dr+N^{2}\int_{0}^{T} \textbf{E}\|\Delta(r, x(r))\|^{2}dr \Big] \\
& \leq 4\|C\|^{2}N^{2}k_{1}\Big[\textbf{E}\|x_{1}\|^{2}+M^{2}\textbf{E}\|\phi(-h)\|^{2}+M^{2}K^{2}\textbf{E}\|\phi(-h)\|^{2}+N^{2}K^{2}L^{2}_{\Delta}T\textbf{E}\|x\|^{2}\Big].
\end{align*}
Using the definition of operator $L_{T}$, we have
\begin{align} \label{3}
I_{3}&\coloneqq \textbf{E}\|L_{T}u\|^{2} = \textbf{E}\|\int_{0}^{t}\mathscr{E}^{A,B}_{h,\alpha,\alpha}(t-h-r)Cu(r)dr\|^{2}  
\leq 4N^{2}\|C\|^{2}\|L_{T}^{*}\|^{2}k_{1} \\
&\times \Big[\textbf{E}\|x_{1}\|^{2}+M^{2}\textbf{E}\|\phi(-h)\|^{2}+ M^{2}K^{2}\textbf{E}\|\phi(-h)\|^{2}
+ N^{2}K^{2}L^{2}_{\Delta}T\textbf{E}\|x\|^{2}\Big] \nonumber.
\end{align}

Using It\^{o}'s isometry, we attain
\begin{align} \label{4}
I_{4}\coloneqq \textbf{E}\| \int_{0}^{t}\mathscr{E}^{A,B}_{h,\alpha,\alpha}(t-h-r)\Delta(r,x(r))dW(r)\|^{2} \leq N^{2}\int_{0}^{T}\textbf{E}\|\Delta(r,x(r))\|^{2}dr 
\leq N^{2}L_{\Delta}^{2}T\textbf{E}\|x\|^{2}.
\end{align}
Substituting \eqref{1}-\eqref{4} into \eqref{main} and using standard computations, we have
\begin{align} \label{final}
\textbf{E}\| (\Phi x) (t)\|^{2} &\leq 4M^{2}\textbf{E}\|\phi(-h)\|^{2}
+4M^{2}K^{2}\textbf{E}\|\phi(-h)\|^{2}\nonumber \\
&+16N^{2}\|C\|^{2}\|L_{T}^{*}\|^{2}k_{1}\textbf{E}\|x_{1}\|^{2} \nonumber \\
&+16N^{2}\|C\|^{2}\|L_{T}^{*}\|^{2}k_{1}M^{2}(1+K^{2})\textbf{E}\|\phi(-h)\|^{2} \\
&+16N^{2}\|C\|^{2}\|L_{T}^{*}\|^{2}k_{1}N^{2}K^{2}L^{2}_{\Delta}T\textbf{E}\|x\|^{2} \nonumber \\
&+4N^{2}L^{2}_{\Delta}T\textbf{E}\|x\|^{2}. \nonumber 
\end{align}
From \eqref{final} and hypothesis (H3), it follows that there exists $0\leq \lambda <1$ and $C_{1}, C_{2} >0$ such that
\begin{equation*}
\textbf{E}\| (\Phi x) (t)\|^{2}  \leq C_{1}+\lambda\textbf{E}\|x_{1}\|^{2}+C_{2}\textbf{E}\|x\|^{2}.
\end{equation*}
 Taking supremum over $[0,T]$ and considering $\Phi(x(T))=x_{1}$, we get

\begin{align}
\sup_{t \in [0,T]}\textbf{E}\| (\Phi x) (t)\|^{2}  \leq \frac{C_{1}}{1-\lambda}+ \frac{C_{2}}{1-\lambda}\textbf{E}\|x\|^{2}.
\end{align} 
Obviously, $\lambda \neq 1$  and this implies that $\Phi$ maps from $H^{2}$ into itself. \\
\textbf{Step 2}. Now, we prove that $\Phi$ is a contraction mapping.
Let $x, y \in  \mathbb{R}^{n}$ for each $t \in [0,T]$, we have 

\begin{align}
\textbf{E}\| (\Phi x) (t)-(\Phi y) (t) \|^{2} &\leq \textbf{E}\| \int_{0}^{t}\mathscr{E}^{A,B}_{h,\alpha,\alpha}(t-h-r)\left[ \Delta(r,x(r))-\Delta(r,y(r))\right] dW(r)\|^{2}\nonumber \\
& \leq N^{2}\int_{0}^{t}\textbf{E}\|\Delta(r,x(r))-\Delta(r,y(r))\|^{2}dr \leq N^{2}L_{\Delta}^{2}T\textbf{E}\|x-y\|^{2} \\
&\leq  \rho \textbf{E}\|x-y\|^{2}\nonumber ,
\end{align}
which together with (H4) implies that $\Phi$ is a contractive mapping on $H^{2}$ and $\Phi$ has a unique fixed point $x(\cdot) \in H^{2}$ with initial condition  $x(t)=\phi(t)$ for $t \in [-h,0]$. Thus, the system \eqref{Problem2} is controllable on $[0,T]$. This completes the proof.
\end{proof}

 \section{Discussion and future work} \label{sec:discus}
  The main contributions of this paper are as follows:
   \begin{itemize}
   	\item introducing a new delayed Mittag-Leffler type function with permutable matrices by means of three-parameter Mittag-Leffler functions;
   	\item deriving stochastic version of variation of constants formula using the delayed Mittag-Leffler type matrix function;
   	\item proving existence and uniqueness results of mild solution and showing coincidence between the integral equation and mild solution of fractional stochastic delay differential equations system; 
   	\item studying complete controllability results for linear and nonlinear fractional stochastic delay dynamical systems with Wiener noise under certain assumptions.
   \end{itemize} 
   
 The advantage of such results is that we have opened the possibility for a cooperative investigation to solve several issues, for instance, combining the methods of this paper to study the control theory, one may solve stability results such as finite time and Ulam-Hyers type stability, and null controllability analogue for linear and nonlinear case of the results of this paper to hold for a class of problems governed by fractional stochastic time-delay differential equations in finite dimensional spaces. On the other hand, we plan to extend our results to a Caputo type time-delay system of fractional stochastic differential equations with nonpermutable matrices in the forthcoming paper.

\end{document}